\documentclass{article}

\usepackage{graphicx} 
\usepackage{amsmath,amssymb,amsthm}
\usepackage[colorlinks=true, allcolors=blue]{hyperref}
\usepackage{algorithm}
\usepackage{algpseudocode}
\usepackage{bbm}
\usepackage[authoryear,round]{natbib}
\usepackage[a4paper, margin=1.2in]{geometry}
\usepackage{xcolor}
\usepackage[affil-it]{authblk}
\usepackage{tikz}
\usetikzlibrary{positioning}
\usepackage{subcaption}

\newtheorem{theorem}{Theorem}
\newtheorem{corollary}{Corollary}
\newtheorem{lemma}{Lemma}
\newtheorem{proposition}{Proposition}
\theoremstyle{definition}
\newtheorem{definition}{Definition}
\newtheorem{assumption}{Assumption}
\theoremstyle{remark}
\newtheorem{remark}{Remark}

\newcommand{\E}{\mathbb{E}}
\newcommand{\PP}{\mathbb{P}}

\newcommand{\kl}[2]{D_{\text{KL}}({#1}\| {#2})}

\newcommand{\var}{\text{Var}}

\newcommand{\x}{\mathcal{X}}

\newcommand{\palg}{p_\text{alg}}

\newcommand{\m}{\mathtt{m}}
\newcommand{\efact}{\varepsilon_\text{fact}}
\newcommand{\tc}{\mathrm{TC}}
\newcommand{\dtc}{\mathrm{DTC}}

\title{Schedule optimization for tau-leaping in masked discrete diffusion}

\author{Cecilia Secchi\thanks{Bocconi University, Department of Decision Sciences, Milan, Italy. \texttt{cecilia.secchi@phd.unibocconi.it} }\;  and Giacomo Zanella\thanks{Bocconi University, Department of Decision Sciences and BIDSA, Milan, Italy. \texttt{giacomo.zanella@unibocconi.it}\\ GZ acknowledges support from the European Research Council (ERC), through StG ``PrSc-HDBayLe'' grant ID 101076564.}
}
\date{}

\begin{document}

\maketitle

\begin{abstract}
Masked diffusions are popular generative models for discrete distributions. Unlike standard autoregressive sampling, they reveal several coordinates in parallel, approximating each block’s joint conditional law by a product of one-coordinate conditionals. 
The resulting procedure, usually called tau-leaping, reduces computational cost but introduces a factorization error ($\efact$), even with perfectly learned predictors.

We study the resulting tradeoff between generative accuracy and computational cost, focusing on how to choose a denoising schedule 
to minimize $\efact$ for a fixed sampling budget. 
To do so, we establish an exact integral representation of $\efact$ separating the schedule from the target’s dependence structure, summarized by a dependence density $\rho$. This representation yields recursive stationarity equations for optimal schedules and allows us to quantify how estimation errors in $\rho$ affect schedule selection.

As the dimension $N$ and sampling budget grow, we characterize the optimal schedule and quantify the cost of random block sizes relative to a deterministic planner. 
We highlight a fundamental dichotomy: if $\rho$ converges uniformly to a strictly positive continuous profile as $N\to\infty$, schedule optimization can only improve the leading constant of $\efact$, while 
if $\rho$ degenerates, schedule optimization can improve the asymptotic order. Examples based on stationary processes and exchangeable mixtures illustrate these regimes.

\end{abstract}

\section{Introduction}

Discrete diffusion models \citep{austin2021structured,shi2024simplified,sahoo2024simple, campbell2022continuous} have recently emerged as a competitive framework for generative modeling on discrete domains, with applications to text, images, music and biological sequences. 
These models inherit their structure from their continuous-space counterparts \citep{ ho2020denoising, song2020score}:  they are built from two processes, a forward noising process that gradually destroys information, and a reverse process that reconstructs samples from noise. 

This paper studies masked discrete diffusion models. 
A data point is a sequence $x=(x^1,\dots,x^N)$ of $N$ tokens drawn from a target distribution $\pi$ on a finite vocabulary. In the forward process, each coordinate is independently replaced by a special mask symbol $\m$ according to a decreasing noise schedule $\alpha_t$ over the time interval $[0,1]$. 
The reverse process starts from the fully masked sequence and progressively reveals masked coordinates according to the reverse schedule
\[
\beta_t := \alpha_{1-t}.
\]
At each reverse step, the model is trained to receive the set $M\subseteq[N]:=\{1,\dots,N\}$ of the currently revealed coordinates and to output, for every unrevealed coordinate $i\notin M$, a one-coordinate conditional distribution
$p_\theta^i(\cdot\mid x^M)$.
The ideal values  of the standard learned time-independent predictors \citep{ou2024your,zheng2024masked,kim2025train} satisfy
\[
p_\theta^i(\cdot\mid x^M)=\pi_i(\cdot\mid x^M),
\qquad M\subseteq[N],\ i\notin M,
\]
where $\pi_i(\,\cdot\mid x^M)$ is the true conditional law of $X^i$ given the coordinates $X^M=x^M$ (see Section \ref{sec:background}).

Exact reverse sampling can be implemented by simulating the underlying continuous-time Markov chain (CTMC) with the Gillespie algorithm \citep{gillespie1977exact}.  In the masked setting, however, exact simulation reveals one coordinate at a time, and therefore requires \(N\) sequential updates; for long sequences this is expensive, as each update requires a model call. 
In practice, one uses the tau-leaping discretization \citep{gillespie2001approximate,campbell2022continuous}: the reverse time interval is divided into sub-intervals, and within each sub-interval several coordinates are revealed in parallel, using conditionals computed at the start of the step.
This produces a sample with a prescribed number \(K\ll N\) of iterations, hence reducing the number of model calls.

The price of this acceleration is a systematic discretization error. 
When several coordinates are revealed simultaneously, the exact reverse law requires the joint conditional distribution of the whole block, whereas tau-leaping replaces it by the product of the one-coordinate conditionals. 
Thus, even with a perfectly trained model, the sampler incurs an error whenever the coordinates revealed in the same block are conditionally dependent given the previously revealed ones. 
We call this intrinsic discretization error the \emph{factorization error}, which we denote  $\efact$.
 
The error is governed by the reverse schedule $\beta$: large increments reveal many coordinates in parallel and reduce the number of model calls, but introduce substantial conditional-independence bias; small increments reduce the bias but require more sequential steps. 
The schedule therefore sets a tradeoff between sampling cost and approximation accuracy. 
Most existing analyses give upper bounds on the error of a prescribed schedule. In this work we study the converse question: choosing the schedule that minimizes $\efact$ for a given budget of \(K\) reverse steps.

\paragraph{Contribution}
We first express the CTMC tau-leaping sampler for masked diffusion (Algorithm \ref{alg:standard-tau-leaping}) in an equivalent formulation, in which a sampling trajectory is represented as an ordered random partition of the coordinate set (Algorithm~\ref{alg:tau-leaping2}).

For this sampler we work with a $\pi$-dependent function $\rho$, the \emph{dependence density}. Up to the factor $N-1$, its value $\rho(u)$ is the average conditional mutual information between two distinct coordinates when each of the others is independently revealed with probability $u$, and it yields the exact representation (Theorem \ref{thm:main})
\[
\efact(\beta)
=
N\sum_{k=1}^K
\int_{\beta_{t_{k-1}}}^{\beta_{t_k}}
(\beta_{t_k}-u)\rho(u)\,du.
\]
The formula separates the effect of the schedule from the dependence structure of the target: the schedule enters only through the weights \((\beta_{t_k}-u)\), while all the relevant information about the target law is contained in \(\rho\).

The main payoff is that the \emph{shape} of $\rho$ determines how strongly the schedule matters. When $\rho_N$ converges to a nondegenerate continuous profile $g$, optimization over fixed smooth schedules can improve the leading constant of $\efact$ but preserves its $N/K$ scaling (Theorem \ref{thm:asym}); the linear schedule is therefore order-optimal. 
This regime includes locally dependent distributions arising from ergodic Markov chains or more generally sufficiently regular stationary processes (Propositions \ref{prop:markov-example}, \ref{prop:stationary-example}).
For exchangeable mixtures, by contrast, the dependence is global and mediated by a latent variable. As a result, $\rho_N$ concentrates near the origin and its limiting profile degenerates (Proposition \ref{prop:exchangeable-example}); in this regime, schedules with finer early steps can change the asymptotic order of the error. 
With $K_N\approx \log N$, we show that the uniform schedule has error $\Omega(N/\log N)$ whereas a geometric schedule has error $O(\log N)$ (Proposition~\ref{prop:comparison-schedules}). 

The representation has three further consequences. First, it yields explicit stationarity equations satisfied by every finite-step optimizer; under a monotonicity condition, these equations characterize the unique optimizer (Corollary~\ref{cor:recursive}).
Second, the large-$N$ limit yields a variational problem with a closed-form leading-order optimizer (Corollary~\ref{cor:asymptotic-optimal-solution}). 
Third, we quantify the asymptotic penalty caused by the block-size randomness intrinsic to tau-leaping, relative to a deterministic planner following the same limiting schedule (Proposition \ref{prop:comparison-disc}).

The total correlation $\tc$ and the dual total correlation $\dtc$ can be expressed through $\rho$, and this result lets us place several existing schedule bounds and heuristics in a common framework (Propositions~\ref{prop:uniform-bound}, \ref{prop:boundprevious-work}).
Finally, since $\rho$ depends on the unknown target, we estimate its conditional mutual information coefficients through an auxiliary information profile. We compare two estimators of this auxiliary profile and derive a stability bound for schedule optimization (Section~\ref{sec:profile-estimators}).

\paragraph{Related work}
In the CTMC formulation of discrete diffusion models, \citet{conforti2025non, dmitriev2026efficient} establish non-asymptotic convergence guarantees for discrete diffusion samplers, including masked dynamics. 
Their analyses upper bound the terminal KL divergence by controlling the distance between path measures, decomposed into initialization, score-approximation, and time-discretization terms.

A second, masked-specific, line of theory integrates out the continuous-time path and works directly with the terminal ordered partition, controlling the discrepancy between the algorithmic and target distributions through the global information-theoretic quantities $\tc$ and $\dtc$.
\citet{li2025breaking} bound the error of any deterministic schedule, giving a step budget of order $K\gtrsim \tc+\dtc$ to reach a fixed accuracy. 
Using two different schedule constructions, \citet{chen2025optimal} and \citet{zhao2026adaptation} each further reduce the step count to $K\gtrsim \tc\log N$ or $K\gtrsim \dtc\log N$.
In the CTMC formulation above, \citet{dmitriev2026efficient} reach the sharper rate $K\gtrsim \min\{\tc,\dtc\}\log N$.

Related to the present work, \citet{lavenant2025error} and \citet{chen2025optimal} derive the information-profile representation of $\efact$ on which our analysis builds. Starting from this result, \citet{chen2025optimal} study the difficulty of learning the optimal schedule, while \citet{lavenant2025error} study the $N\to\infty$ scaling limit of the schedule-optimization problem for deterministic block sizes.

The practical importance of the masking schedule is well recognized. Cosine-type and other monotone increasing schedules are standard in masked diffusion implementations \citep{shi2024simplified,zhang2025cosine}. In adaptive or online scheduling, \citet{ben2025accelerated} and \citet{foresti2026improved} build schedules on the principle that each parallel update should contribute a roughly constant amount of information, a heuristic counterpart to the error-optimal recursion of Corollary~\ref{cor:recursive}; \citet{cai2026confidence} provide theoretical guarantees for algorithms of this type.

\paragraph{Organization}
Section~\ref{sec:background} reviews masked diffusion models and the tau-leaping sampler.
Section~\ref{sec: main result} introduces the conditional mutual information profile and the dependence density, states the exact integral representation of the factorization error, derives the recursive optimality condition, and establishes upper bounds.
Section~\ref{sec:asymptotic} studies the asymptotic variational limit and the resulting optimal smooth schedules.
Section~\ref{sec:examples} computes the dependence density for several families of target distributions.
Section~\ref{sec:profile-estimators} studies estimation of $\iota$ and $\rho$ through the auxiliary profile $f$, and quantifies the effect of estimation errors on schedule optimization.
Technical proofs are collected in the appendix.

\paragraph{Concurrent work.} 
Concurrent and independent work by \citet{wainwright2026data}, which appeared online during the final writing phase of the present manuscript, derives a similar integral representation for the factorization error, data-dependent schedule bounds and a fine-partition analysis.

\paragraph{Notation.}
Unless otherwise stated, $N\ge2$ and the finite vocabulary $\x$ has size $|\x|=L\ge2$. Let $\x^*:=\x\cup\{\m\}$ be the augmented vocabulary, where $\m$ denotes the mask symbol.
The elements of $(\x^*) ^{N}$ are written as $x=(x^1,\dots,x^N)\in(\x^*) ^{N}$. For any subset $M\subseteq[N]$, we denote by $x^M=(x^i)_{i\in M}\in (\x^*)^{|M|}$ the vector consisting of the entries of $x$ indexed by $M$. 

Let $\pi$ be the target distribution on $\x^N$.  For any subset $M \subseteq [N]$, let $\pi_M$ denote the marginal distribution of $x^M$ under $\pi$. For each $i \in [N]\setminus M$, define the conditional distribution $\pi_i(\cdot\mid x^M)=\pi(x^i=\cdot\mid x^M)$. 

\section{Masked diffusions, tau-leaping and ordered partitions}
\label{sec:background}

Adopting the continuous-time Markov chain (CTMC) perspective, we first review the construction of masked diffusion models and the standard sampling procedure based on the tau-leaping approximation \citep{shi2024simplified,sahoo2024simple,austin2021structured, kim2025train}.
We then give an equivalent formulation of the generative process in which a sampling trajectory is represented by an ordered random partition of the coordinate set; this representation is the starting point for our analysis of the factorization-error.

\paragraph{Forward process.} 
Let $\alpha : [0,1] \to [0,1]$ be a continuous, strictly decreasing noise schedule with $\alpha_0 = 1$ and $\alpha_1 = 0$. 
Starting from a data point $x_0 \sim \pi$, the forward process masks the coordinates independently according to $\alpha$ as follows. At time $t \in [0,1]$, coordinate $i$ is left unchanged with probability $\alpha_t$ and is replaced by $\m$ with probability $1-\alpha_t$. 
Therefore, the conditional distribution of $x_t$ given $x_0$ factorizes as
\[
q_{t\mid 0}(x_t\mid x_0)=\prod_{i=1}^{N} q_{t\mid 0}^{i}(x_t^i\mid x_0^i),
\qquad
q_{t\mid 0}^{i}(x_t^i\mid x_0^i)=\mathrm{Cat}\!\left(\alpha_t e_{x_0^i}+(1-\alpha_t)e_{\m}\right).
\]
Here $e_j \in \mathbb{R}^{L+1}$ denotes the one-hot vector associated with token $j$, and $\mathrm{Cat}(\mu)$ the categorical distribution on $\x^*$ with probability vector $\mu$. 
Since $\alpha_1=0$, the terminal state is almost surely fully masked. 

\paragraph{Reverse process}
The reverse process starts from the fully masked state $x_1=(\m,\dots,\m)\in(\x^*)^N$ and generates a sample from $\pi$ at time $0$. 
Unlike the forward process, the reverse dynamics generally couple the coordinates, so exact simulation proceeds through sequential single-coordinate updates.
To enable parallel sampling,
for $0\le s<t\le 1$ the joint reverse transition $q_{s\mid t}(\cdot\mid x_t)$ is replaced by the product of its one-coordinate marginals.
For each fixed $x_t$, this product is the optimal factorized approximation of the true joint conditional, in the sense that it minimizes $p\mapsto \mathrm{KL}\bigl(q_{s\mid t}(\cdot\mid x_t)\,\|\,p\bigr)$ over all product distributions $p$ \citep[Thm.~4.2 and App.~A]{lou2023discrete}.

Given $x_t\in(\x^*)^N$, let $U_t:=\{j\in[N]:x_t^j\neq \m\}$ denote the set of unmasked coordinates, 
and write $x_t^{U_t}$ for the corresponding values.
For $0\le s<t\le 1$, the one-coordinate marginals of the true reverse law $q_{s\mid t}(x_s\mid x_t)$ are given by \citep[App.~D.2]{ou2024your}:
\begin{equation}\label{eq:revprocess}
q_{s\mid t}^i(\,\cdot\mid x_t)
=
\begin{cases}
\delta_{x_t^i},
& x_t^i\neq \m,\\
\mathrm{Cat}\!\left(
\frac{1-\alpha_s}{1-\alpha_t}e_{\m}
+
\frac{\alpha_s-\alpha_t}{1-\alpha_t}
\sum_{x\in\x}
\pi_i(x\mid x_t^{U_t})e_x
\right),
& x_t^i=\m,
\end{cases}
\end{equation}
where $\pi_i(x\mid x_t^{U_t}) = \pi(x_0^i=x\mid x_0^{U_t}=x_t^{U_t})$
is the conditional law, under $\pi$, of the clean token at coordinate $i$ given that the coordinates $x_t^{U_t}$ are observed. 
In practice, the unknown conditional $\pi_i(\cdot\mid x_t^{U_t})$ is replaced by a learned, time-independent predictor $p_\theta^i(\cdot\mid x_t^{U_t})$, which is a distribution over clean tokens and assigns zero mass to $\m$.
The sampler thus evolves according to the factorized kernel
\begin{equation}\label{eq:reverse}
\hat q_{s\mid t}^{\theta}(x_s\mid x_t)
=
\prod_{i=1}^N
\hat q_{s\mid t}^{\theta,i}(x_s^i\mid x_t),
\qquad
\hat q_{s\mid t}^{\theta,i}(\,\cdot\mid x_t)
=
\begin{cases}
\delta_{x_t^i},
& x_t^i\neq \m,\\
\mathrm{Cat}\!\left(
\frac{1-\alpha_s}{1-\alpha_t}e_{\m}
+
\frac{\alpha_s-\alpha_t}{1-\alpha_t}
\sum_{x\in\x}
p_\theta^i(x\mid x_t^{U_t})e_x
\right),
& x_t^i=\m.
\end{cases}
\end{equation}
This parallel update rule is the Tweedie tau-leaping approximation of the exact reverse dynamics \citep{lou2023discrete}.

\subsection{Tau-leaping as a random ordered partition}
Given the predictors $p_\theta^i$, a time grid $0=t_0<t_1<\cdots<t_K=1$ and a continuous, strictly increasing denoising schedule $\beta:[0,1]\to[0,1]$, with $\beta_0=0$ and $\beta_1=1$, define a sampling scheme.
Equivalently, $\beta$ is determined by the forward noise schedule above through the relation $\beta_t=\alpha_{1-t}$.
Set 
\[
\beta_k:=\beta_{t_k} \qquad q_k:=\frac{\beta_k-\beta_{k-1}}{1-\beta_{k-1}}
\qquad \text{ for } k=1,\dots,K.
\]
The standard reverse tau-leaping sampler in Algorithm~\ref{alg:standard-tau-leaping} is obtained from \eqref{eq:reverse} with $\alpha_t=\beta(1-t)$ by taking $(s,t)=(1-t_k,1-t_{k-1})$ at the $k$-th step; we write $y_k$ for its state after $k$ steps.
At each step, every currently masked coordinate is independently selected for unmasking with probability $q_k$; if selected, its value is sampled from the one-coordinate predictor $p_\theta^i(\cdot\mid y_{k-1})$.
For a state $y$ with unmasked set $U$, we use $p_\theta^i(\cdot\mid y)$ as shorthand for $p_\theta^i(\cdot\mid y^U)$.

\begin{algorithm}[H]
\caption{Standard reverse tau-leaping sampler}
\label{alg:standard-tau-leaping}
\begin{algorithmic}[1]
\Require schedule $0=\beta_0<\beta_1<\cdots<\beta_K=1$, predictor $p_\theta$
\State Initialize $y_0\gets(\m,\dots,\m)$
\For{$k=1$ to $K$}
    \For{$i=1$ to $N$}
        \If{$y_{k-1}^i=\m$}
            \State Sample $y_k^i\sim\mathrm{Cat}\bigl((1-q_k)\,e_\m+q_k\,p_\theta^i(\cdot\mid y_{k-1})\bigr)$
        \Else
            \State $y_k^i\gets y_{k-1}^i$
        \EndIf
    \EndFor
\EndFor
\State \Return $y_K$
\end{algorithmic}
\end{algorithm}
 
For the analysis, it is convenient to turn to an equivalent formulation of the same sampler, as done in   \cite{ben2025accelerated,li2025breaking,lavenant2025error, chen2025optimal}.
Algorithm~\ref{alg:tau-leaping2} makes explicit that tau-leaping samples a random ordered partition: the draws of $s_k$ and $z_k$ depend on the past only through $z_{<k}$, and not on the generated token values.
Hence the entire block sequence $z$ may equivalently be sampled in advance, after which the token values are generated block by block.
For an integer $0\le s\le|A|$, let $\mathrm{Unif}(A;s)$ denote the uniform distribution over the subsets of $A$ of cardinality $s$.

\begin{algorithm}[H]
\caption{Reverse tau-leaping sampler, second form}
\label{alg:tau-leaping2}
\begin{algorithmic}[1]
\Require schedule $0=\beta_0<\beta_1<\cdots<\beta_K=1$, predictor $p_\theta$
\State Initialize $y_0\gets(\m,\dots,\m)$ and $z_{<1}\gets\varnothing$
\For{$k=1$ to $K$}
    \State Sample $s_k\sim\mathrm{Bin}\bigl(N-|z_{<k}|,\,q_k\bigr)$
    \State Sample $z_k\sim\mathrm{Unif}\bigl([N]\setminus z_{<k};\,s_k\bigr)$
    \State Sample $y_k^{z_k}\sim\prod_{i\in z_k}p_\theta^i(\cdot\mid y_{k-1}^{z_{<k}})$
    \State Set $z_{<k+1}\gets z_{<k}\cup z_k$
    \State Set $y_k^{z_{<k}}\gets y_{k-1}^{z_{<k}}$ and $y_k^{i}\gets\m$ for all $i\notin z_{<k+1}$
\EndFor
\State \Return $y_K$
\end{algorithmic}
\end{algorithm}

\begin{proposition}[Equivalence of tau-leaping implementations]
\label{prop:tau-leaping-equivalence}
For every schedule $0=\beta_0<\beta_1<\cdots<\beta_K=1$ and every predictor $p_\theta$, Algorithms~\ref{alg:standard-tau-leaping} and~\ref{alg:tau-leaping2} induce the same law on the trajectory $(y_0,\dots,y_K)$ and, in particular, on the terminal output $y_K$.
\end{proposition}
The proof is given in Appendix~\ref{app:proof section equivalence}.

Let $\nu^\beta$ denote the law of $z$ defined in Algorithm~\ref{alg:tau-leaping2}. It factorizes as
\[
\nu^\beta(z)=\prod_{k=1}^K\nu^\beta(z_k\mid z_{<k}), \qquad \nu^\beta(z_k\mid z_{<k})=q_k^{\,|z_k|}(1-q_k)^{\,N-|z_{<k}|-|z_k|}
\]
for every admissible $z_k\subseteq[N]\setminus z_{<k}$, and the conditional probability is zero otherwise.
We use the convention $0^0=1$.
Since $\beta_K=1$ we have $q_K=1$, so every coordinate still masked before the final step is revealed at step $K$; hence $(z_1,\dots,z_K)$ is an ordered partition of $[N]$ with possibly empty blocks: the $z_k$ are pairwise disjoint and $\bigcup_{k=1}^{K}z_k=[N]$.

Conditionally on the ordered partition $z$, the model generates the terminal sample $x:=y_K\in\x^N$ block by block.
Because revealed coordinates are never modified, $y_{k-1}^{z_{<k}}=x^{z_{<k}}$, and the conditional law of $x$ given $z$ is
\[
p_\theta(x;z)
:=\prod_{k=1}^{K}p_\theta\bigl(x^{z_k}\mid x^{z_{<k}}\bigr),
\qquad
p_\theta\bigl(x^{z_k}\mid x^{z_{<k}}\bigr)
:=\prod_{i\in z_k}p_\theta^i\bigl(x^i\mid x^{z_{<k}}\bigr),
\]
with the convention that conditioning on $x^{z_{<1}}=x^{\varnothing}$ means conditioning on the fully masked input.
The joint law of the output and of the ordered partition is therefore
\[
\palg(x,z)=p_\theta(x;z)\,\nu^\beta(z),
\]
and the law of the sampler output is its $x$-marginal,
\[
\palg(x)=\sum_z p_\theta(x;z)\,\nu^\beta(z).
\]

\subsection{Factorization error}

The denoising schedule $\beta$ determines how coordinates are allocated across parallel update blocks. To study its effect on sampling accuracy, we consider an upper bound on the KL divergence between the target distribution and the marginal output distribution.

Following the literature (see e.g. \citealp{ben2025accelerated,li2025breaking,lavenant2025error, chen2025optimal}), we define the \emph{learning error} and the \emph{factorization error}, respectively, by
\begin{align}
\varepsilon_{\mathrm{learn}}
&:=
\mathbb E_{\pi(x)\nu^\beta(z)}
\sum_{k=1}^K
\sum_{i\in z_k}
\log
\frac{
\pi_i(x^i\mid x^{z_{<k}})
}{
p_\theta^i(x^i\mid x^{z_{<k}})
},
\label{eq:learning-error}
\\
\varepsilon_{\mathrm{fact}}
&:=
\mathbb E_{\pi(x)\nu^\beta(z)}
\sum_{k=1}^K
\log
\frac{
\pi(x^{z_k}\mid x^{z_{<k}})
}{
\prod_{i\in z_k}
\pi_i(x^i\mid x^{z_{<k}})
}.
\label{eq:factorization-error}
\end{align}
The chain rule for KL divergence, together with its monotonicity under marginalization, yields
\begin{equation}\label{eq:kl-decomposition}
\kl{\pi(x)}{\palg(x)}
\le
\kl{\pi(x,z)}{\palg(x,z)}
=
\varepsilon_{\mathrm{learn}}
+
\varepsilon_{\mathrm{fact}},
\end{equation}
where $\pi(x,z):=\pi(x)\nu^\beta(z)$ and $\palg(x,z):=p_\theta(x;z)\nu^\beta(z)$.

Both terms are nonnegative: $\efact$  is an average of conditional total correlations, while $\varepsilon_{\mathrm{learn}}$ is an average of one-coordinate conditional KL divergences. The learning error therefore vanishes when all one-coordinate conditionals are learned exactly:
\[
p_\theta^i(\cdot\mid x^M)=\pi_i(\cdot\mid x^M),
\qquad M\subseteq[N],\quad i\notin M.
\]
By contrast, the factorization error generally persists even under perfect learning. Indeed it is a purely algorithmic error induced by parallel updates and quantifies the discrepancy introduced by approximating each joint block conditional distribution by the product of its one-coordinate conditional marginals,
\[
\pi(x^{z_k}\mid x^{z_{<k}})
\approx\prod_{i\in z_k}\pi_i(x^i\mid x^{z_{<k}}).
\]
Under perfect learning, $\efact$ equals the joint KL divergence in \eqref{eq:kl-decomposition} and thus provides an upper-bound on the terminal KL divergence. This motivates our objective: for a fixed budget of $K$ steps, we seek a schedule that minimizes the factorization error,
\begin{equation}\label{eq:kl-target}
\min_{0=\beta_0\le\beta_1\le\cdots\le\beta_K=1}
\varepsilon_{\mathrm{fact}}(\beta).
\end{equation}

\section{Dependence densities and optimal schedules} \label{sec: main result}

We introduce the definitions needed to express the factorization error as an integral functional of the schedule.

To quantify the conditional dependence underlying this error, let $\iota(i)$ denote the average dependence between two unrevealed coordinates after observing $i$ uniformly selected coordinates:
\begin{equation}\label{eq:cmi-profile}
\iota(i):=\E_\sigma\!\left[
I_\pi\!\left(X^{\sigma_{i+1}};X^{\sigma_{i+2}}\mid X^{\sigma_{\le i}}\right)
\right]\ge0\qquad i=0,\dots,N-2,
\end{equation}
where $\sigma$ is an independent, uniformly random permutation of $[N]$, $\sigma_{\le i}:=\{\sigma_1,\ldots,\sigma_i\}$, with $\sigma_{\le0}=\varnothing$. For fixed, distinct coordinate indices $a,b$ and a set $A$ disjoint from them, $I_\pi(X^a;X^b\mid X^A)$ denotes the standard conditional mutual information under $\pi$.

\begin{definition}[Dependence density]\label{def:dependence-density}
For a distribution $\pi$ on $\x^N$, define
\begin{equation}\label{eq:bernstein_rep}
\rho(u):=(N-1)\sum_{i=0}^{N-2}\iota(i)B_i^{N-2}(u),
\qquad u\in[0,1],
\end{equation}
where $B_i^n(u):=\binom ni u^i(1-u)^{n-i}$ is the $i$-th Bernstein basis polynomial of degree $n$, and $\iota$ is given by \eqref{eq:cmi-profile}.
\end{definition}
Equivalently,
\begin{equation}\label{eq:rho-binomial}
\rho(u)=(N-1)\E[\iota(B_u)],
\qquad B_u\sim\mathrm{Bin}(N-2,u).
\end{equation}
Thus $\rho(u)/(N-1)$ is the average conditional mutual information between two uniformly selected distinct coordinates when each of the other $N-2$ coordinates is independently revealed with probability $u$. It records how conditional dependence varies with the reveal probability and is not normalized to integrate to one.

Since $\iota(i)\ge0$, the function $\rho$ is nonnegative. The next lemma, proved in Appendix~\ref{app:section:main result}, characterizes when it vanishes.
\begin{lemma}\label{lem:product-measure}
If $\pi$ is a product measure, then $\rho\equiv 0$; otherwise $\rho(u)>0$ for every $u\in(0,1)$.
\end{lemma}

The following theorem gives an exact integral representation of the factorization error for the tau-leaping sampler. Its proof, presented in Appendix~\ref{app:section:main result}, builds on the implicit factorization-error representation developed independently by \citet{lavenant2025error} and \citet{chen2025optimal}, which we recall there in terms of conditional mutual information.

\begin{theorem}[Representation of $\efact$]\label{thm:main}
The factorization error defined in Equation \eqref{eq:factorization-error} for Algorithm \ref{alg:tau-leaping2} admits the representation
\begin{equation}\label{eq:mainresult}
            \varepsilon_\text{fact}(\beta)
        =N\sum_{k=1}^K \int_{\beta_{k-1}}^{\beta_k} (\beta_k-u) \rho(u)du.
\end{equation}
\end{theorem}

This representation yields a recursive formula for the optimal schedule $\beta$ that minimizes the factorization error. 

\begin{corollary}\label{cor:recursive}
If $\pi$ is not a product measure, every minimizer of $\varepsilon_\text{fact}(\beta)$ over the simplex
\[
\{(\beta_1,\ldots,\beta_{K-1}):0\le\beta_1\le\cdots\le\beta_{K-1}\le1\}
\]
lies in its interior.
Moreover every minimizer satisfies the recursion
\begin{equation}\label{eq:recursive formula}
\beta_{k+1}
=
\beta_k+\Psi(\beta_k,\beta_k-\beta_{k-1}),
\quad 
\Psi(b,\delta)
:=
\frac{1}{\rho(b)}
\int_{b-\delta}^b \rho(u)\,du 
\quad k=1,\dots,K-1.        
\end{equation}
If $b\mapsto\Psi(b,\delta)$ is nondecreasing for every fixed $\delta$, then the minimizer is unique. 

If $\pi$ is a product measure, then
$\efact(\beta)=0$ for every schedule, so every schedule is optimal.
\end{corollary}

\begin{remark}[On uniqueness]\label{rem:uniqueness}
The monotonicity condition of Corollary~\ref{cor:recursive} holds, for example, whenever $\log\rho$ is concave.

The recursion \eqref{eq:recursive formula} determines the entire schedule from the initial value $\beta_1$ and stationary schedules correspond to the values of $\beta_1$ for which the recursion terminates exactly at $\beta_K=1$. 
Under the monotonicity assumption the resulting shooting map is monotone and the unique minimizer can be computed by bisection on $\beta_1$.
Without this assumption, the shooting equation may have several roots. In that case, one must identify all relevant stationary schedules, or use a global finite-dimensional optimization procedure, and compare their objective values.
For instance, when $K=2$, the stationary condition reduces to
\[
\int_0^x \rho(u)\,du=(1-x)\rho(x),
\qquad x=\beta_1.
\] 
For the profile $\rho(u)=2+\sin(30u)$,  this becomes
\[
2x+\frac{1-\cos(30x)}{30}-(1-x)(2+\sin(30x))=0,
\]
which has three roots in $(0,1)$; the first-order conditions alone therefore do not identify the minimizer. This $\rho$ is not of the polynomial form \eqref{eq:bernstein_rep} and serves just as an illustration.
\end{remark}

\begin{remark}[Target-agnostic schedules]\label{prop:heuristics}
The optimal schedule requires full knowledge of $\rho$, which is generally unavailable in applications.
It is therefore useful to consider simple target-agnostic schedules adapted to qualitative shapes of $\rho$ such as whether the mass of $\rho$ concentrates near the origin.
    
For example, the geometric schedule $\beta_{k+1}= (1+a)\beta_k$  
studied in Proposition~\ref{prop:boundprevious-work}, formally solves the local optimality recursion for the idealized profile \(\rho(u)\propto u^{-2}\): every stationarity relation holds except the one adjacent to the origin where the profile diverges (the proof is given in Appendix~\ref{app:section:main result}).
No dependence density realizes this profile exactly (by \eqref{eq:bernstein_rep}, $\rho$ is a polynomial, hence bounded) but it captures the behavior of genuine targets such as the exchangeable mixtures of Section~\ref{sec:examples}.
\end{remark}

\subsection{Bounds for target-agnostic schedules}  \label{sec:upper bounds}

Several previous works bound the KL divergence between the algorithm distribution and the data distribution in terms of dependence functionals of $\pi$ \citep{chen2025optimal, lavenant2025error, li2025breaking, dmitriev2026efficient, zhao2026adaptation}.
We show that these quantities admit simple representations in terms of the dependence profile $\rho$.

We use the standard notation $H(\,\cdot\,)$, $H(\,\cdot\mid\cdot\,)$, and
$I(\,\cdot\,;\,\cdot\,)$ for entropy, conditional entropy and mutual information. Throughout, $X=(X^1,\dots,X^N)\sim\pi$ and $X^{-i}=(X^j)_{j\neq i}$.

\begin{definition}
The \emph{total correlation} and \emph{dual total correlation} of $\pi$ are defined, respectively, as
\begin{align*}
    TC(\pi)&:=\sum_{i=1}^N H(X^i)-H(X^1,\dots,X^N)\\
DTC(\pi)&:=H(X^1,\dots,X^N)-\sum_{i=1}^N H(X^i\mid X^{-i})\,.
\end{align*}
Also, define their normalized sum
\begin{align*}
    D(\pi):&=\frac{TC(\pi)+DTC(\pi)}{N}=\frac{1}{N}\sum_{i=1}^N I(X^i;X^{-i}).
\end{align*}
\end{definition}

The following Lemma (proved in Appendix \ref{app:proof upper bounds}) relates these quantities to the dependence density.
\begin{lemma}\label{lemma:correlation}
The quantities $D(\pi)$, $TC(\pi)$, and $DTC(\pi)$ can be expressed in terms of $\rho$ as
\begin{equation*}
\mathrm{D}(\pi)=\int_0^1\rho(u)\,du,\qquad \tc(\pi)=N\int_0^1 (1-u)\rho(u)\,du,\qquad
\dtc(\pi)=N\int_0^1u\,\rho(u)\,du.
\end{equation*}
Equivalently, 
\[
D(\pi)=\sum_{i=0}^{N-2}\iota(i), \qquad \tc(\pi)=\sum_{i=0}^{N-2}(N-i-1)\iota(i), \qquad \dtc(\pi)=\sum_{i=0}^{N-2}(i+1)\iota(i);
\]
in particular $D(\pi)\le\min\{\tc(\pi),\dtc(\pi)\}$.
\end{lemma}

As a consequence of Theorem~\ref{thm:main} and the decomposition \eqref{eq:kl-decomposition}, we obtain a universal upper bound on the terminal KL divergence $\kl{\pi}{\palg}$ in terms of $D(\pi)$. 
Thanks to the framework of Algorithm \ref{alg:tau-leaping2}, the following result can be compared not only with the works we cited above, but also with results obtained from a CTMC perspective, namely \citet[Theorem 3.1.1]{conforti2025non} and \citet[Theorem 3]{dmitriev2026efficient}. 
In contrast to path-space CTMC analyses, this bound is stated directly at the level of the terminal distribution and separates the learning and factorization terms appearing in \eqref{eq:kl-decomposition}.

\begin{proposition}\label{prop:uniform-bound}
For any schedule $\beta$, writing $\Delta\beta_{\max}=\max_k\Delta\beta_k$,
\begin{align}\label{D beta_max}
    \kl{\pi(x)}{\palg(x)}
&\le
N\Delta\beta_{\max}D(\pi)+
\varepsilon_{\mathrm{learn}}
.
\end{align}
For the constant step size $\Delta\beta_k=1/K$ this becomes 
\[
\kl{\pi(x)}{\palg(x)}\le\frac{ND(\pi)}{K}+
\varepsilon_{\mathrm{learn}}
.
\]
\end{proposition}

A bound independent of the target distribution follows from the universal estimate $D(\pi)\le\log|\x|$, which holds since $I(X^i;X^{-i})\le H(X^i)\le\log|\x|$ for every $i$.

\begin{remark}
Since  $\E[s_k]=N\Delta\beta_k$, the factor in \eqref{D beta_max} is the maximum expected block size, $N\Delta\beta_{\max}=\max_k\mathbb E[s_k]$. 
The bounds of \citet[Theorem 7]{lavenant2025error} and \citet[Theorem 1]{li2025breaking} instead carry the expected maximum block size, in the form $(\E[\max_k s_k]-1)D(\pi)$. Since by Jensen $\max_k\E[s_k]\le\E[\max_k s_k]$, the leading factor in \eqref{D beta_max} is sharper when $\E[\max_k s_k]\ge\max_k\E[s_k]+1$, that is, whenever the block sizes fluctuate by at least one token on average. 
\end{remark}

Whereas \eqref{D beta_max} holds for an arbitrary schedule, one can also derive guarantees tailored to specific prescribed schedules, in the spirit of \citet[Theorem 1.9]{chen2025optimal} and \citet[Corollary 2]{dmitriev2026efficient}, by working directly with the decomposition of Equation~\eqref{eq:Nfact} in the proof of Theorem~\ref{thm:main}.

\begin{proposition} \label{prop:boundprevious-work}
Let $a>0$ with $a/N\in(0,1)$, and let $K:= 1+\left\lceil\tfrac{\log(N/a)}{\log(1+a)}\right\rceil$.
Consider the geometric schedule 
\[
\beta_0=0,\qquad \beta_k=\min\{1,\,(1+a)^{k-1}a/N\},\qquad k=1,\dots,K,
\]
for which $\beta_K=1$. Then
\[
\varepsilon_{\mathrm{fact}}\le a\bigl(\dtc(\pi)+D(\pi)\bigr)\le 2a\,\dtc(\pi).
\]
Symmetrically, consider the reverse geometric schedule, whose residual distance to the endpoint decays geometrically,
\[
\beta_K=1,\qquad 1-\beta_k=\min\{1,\,(1+a)^{K-1-k}a/N\},\qquad k=0,\dots,K-1,
\]
for which $\beta_0=0$ and $1-\beta_{K-1}=a/N$. Then
\[
\varepsilon_{\mathrm{fact}}\le a\bigl(\tc(\pi)+D(\pi)\bigr) \le 2a\,\tc(\pi).
\]
\end{proposition}
The proof is in Appendix \ref{app:section:main result}

\begin{remark}
For the geometric schedule, $\varepsilon_{\mathrm{fact}}\le\varepsilon$ is guaranteed by the choice $a=\varepsilon\,(2\,\dtc(\pi))^{-1}$, assuming $a\le1$. Since the number of blocks satisfies $K\le 1+\left\lceil \log(N/a)/\log(1+a)\right\rceil$, for $a\le1$ we have $K\lesssim a^{-1}\log(N/a)$, and therefore
\[
K\lesssim \frac{\dtc(\pi)}{\varepsilon}\log\frac{N\,\dtc(\pi)}{\varepsilon}.
\]

\end{remark}

\section{Asymptotic limit}\label{sec:asymptotic}
In this section, we study sequences of distributions $\{\pi_N\}_{N\geq 1}$ whose dependence densities $\rho_{N}$ converge uniformly to a limiting profile $g$. 
We derive a leading-order expression for $\efact$, identify the schedule minimizing the resulting variational functional, and quantify the excess cost of tau-leaping relative to a planner with deterministic block sizes following the same limiting schedule.
The proofs can be found in Appendix \ref{proof:asymptotic}.
Throughout, $\rho_N$ and $\iota_N$ denote the dependence density and the conditional mutual information profile of $\pi_N$, respectively.

\begin{assumption}\label{ass:G_N-asymp}
For a sequence $\{\pi_N\}_{N\ge 2}$ there exists a continuous function $g:[0,1]\to\mathbb{R}_+$ such that
$\|\rho_N-g\|_\infty
\underset{N\to\infty}{\longrightarrow}0$.
\end{assumption}
Let $D_\infty:=\int_0^1g(u)\,du$. Assumption~\ref{ass:G_N-asymp} and Lemma~\ref{lemma:correlation} imply $D(\pi_N)\to D_\infty$.

\begin{lemma}\label{lem:uniform-rho}
Suppose there exists a continuous function $g:[0,1]\to\mathbb{R}_+$ such that, for $N\ge3$,
\begin{equation}\label{eq:coefficient-convergence}
\max_{0\le j\le N-2}\left|(N-1)\iota_{N}(j)-g\!\left(\frac{j}{N-2}\right)\right|
\underset{N\to\infty}{\longrightarrow}0.
\end{equation}
Then $\|\rho_N(u)-g(u)\|_\infty\to0$ as $N\to\infty$.
\end{lemma}
The proof combines the Bernstein representation \eqref{eq:bernstein_rep} with Bernstein's approximation theorem. 

\begin{theorem}\label{thm:asym}
Let $\{\pi_N\}_{N\ge 2}$ satisfy Assumption \ref{ass:G_N-asymp} and let $\beta\in\mathcal{C}^1([0,1])$ be strictly increasing with $\beta(0)=0$ and $\beta(1)=1$. For each $N,K$, consider the schedule $\beta_k=\beta(k/K)$, $k=0,\dots,K$, and write $\varepsilon^{N,K}_\text{fact}(\beta)$ for the corresponding factorization error. Then, as $N,K\to\infty$ jointly and with no relation required between them,
\begin{equation}\label{eq:asympt epsilon}
\varepsilon^{N,K}_\text{fact}(\beta)
=
\frac{N}{2K}\int_0^1 g(\beta(t))\,\beta'(t)^2\,dt
+o\!\left(\frac{N}{K}\right).
\end{equation}
\end{theorem}

Equation \eqref{eq:asympt epsilon} yields a variational  problem over smooth increasing schedules, whose minimizer is explicit.

\begin{corollary}\label{cor:asymptotic-optimal-solution}
 Under the hypotheses of Theorem \ref{thm:asym}, assume $g(u)>0$ for $u\in[0,1]$, the schedule minimizing the leading term in \eqref{eq:asympt epsilon} satisfies
\[
\beta'(t)\propto\frac{1}{\sqrt{g(\beta(t))}},
\quad\text{that is,}
\quad
\beta^*(t)=G^{-1}\bigl(t\,G(1)\bigr),
\quad G(y):=\int_0^y\sqrt{g(u)}\,du,
\]
with corresponding factorization error 
\[
\varepsilon^{N,K}_\text{fact}(\beta^*)
=
\frac{N}{2K}\left(\int_0^1 \sqrt{g(u)}\,du\right)^2
+o\!\left(\frac{N}{K}\right).
\]
\end{corollary}
The variational problem and its solution are those of \citet[Prop.~13]{lavenant2025error}; we include the statement for completeness, the novelty here being the identification of the limit \eqref{eq:asympt epsilon} for the tau-leaping sampler.

\begin{remark}\label{remark:optVSlin}
For the linear schedule $\beta(t)=t$, Theorem \ref{thm:asym} gives 
\[
\varepsilon_\text{fact}^\text{lin}=\frac{N}{2K} \int_0^1 g(u)du+o\left(\frac{N}{K}\right).
\]
Under the hypotheses of Corollary~\ref{cor:asymptotic-optimal-solution}, write $\varepsilon_\text{fact}^\text{opt}:=\varepsilon_\text{fact}^{N,K}(\beta^*)$ for the error of its optimal limiting smooth schedule. Then
\[
\frac{\varepsilon_\text{fact}^\text{lin}}{\varepsilon_\text{fact}^\text{opt}}\underset{N,K\to\infty}{\longrightarrow} \frac{\int_0^1 g(u)du}{\left(\int_0^1 \sqrt{g(u)}du\right)^2}\geq 1,
\]
where the inequality follows from Jensen.
Thus, under these hypotheses, optimizing the limiting functional over fixed smooth schedules improves the leading constant but not the $N/K$ scaling.
The same conclusion does not necessarily apply outside this regime, for example when $g$ degenerates as in the exchangeable mixtures case of Section~\ref{sec:examples}.
\end{remark}

\subsection{The cost of random block sizes}

Given a schedule $\beta$, tau-leaping reveals a \emph{random} number of coordinates at each step: as shown in the proof of Theorem~\ref{thm:main}, the joint vector of block sizes follows a multinomial distribution, so that $s_k\sim\mathrm{Bin}(N,\Delta\beta_k)$ marginally.
We compare the asymptotics of Theorem~\ref{thm:asym} with those obtained by \citet{lavenant2025error} for a planner that uses a uniformly random coordinate ordering and deterministic block sizes
\[
s_k^{\mathrm{det}}
=
\lceil N\beta(k/K)\rceil
-
\lceil N\beta((k-1)/K)\rceil,
\qquad k=1,\dots,K.
\]
For the same limiting schedule $\beta$, this comparison quantifies the excess factorization cost caused by the random block sizes of tau-leaping. 

For the comparisons in this subsection, assume the stronger coefficient convergence \eqref{eq:coefficient-convergence} and $D_\infty>0$. These hypotheses ensure that the deterministic limits of \citet{lavenant2025error} apply.

When $N,K\to\infty$ with $N/K\to\infty$, the two factorization costs coincide to first order by Theorem~\ref{thm:asym} and \citet[Theorem~12]{lavenant2025error}.
This is due to the fact that the random block sizes concentrate around their deterministic counterparts
$s_k^{\det}\approx N\Delta\beta_k$: for smooth schedules $\beta$ with $\inf \beta'>0$,  whenever $N\Delta\beta_k\to\infty$ we have 
$
\tfrac{s_k}{N\Delta\beta_k}\longrightarrow 1$ in probability.

The difference becomes visible once the expected block sizes stay bounded. When $N/K\to\bar s\in[1,\infty)$, Theorem~\ref{thm:asym} gives the tau-leaping cost
\[
\efact^{N,K}
=
\frac{\bar s}{2}\int_0^1 g(\beta(t))\,\beta'(t)^2\,dt+o(1),
\]
while, if $\beta'$ has finitely many local extrema, the deterministic limit of \citet[Theorem 15]{lavenant2025error} is
\[
\efact^{\det,N,K}(\beta)
=
\frac{1}{2}\int_0^1 g(\beta(t))\bigl[\bar s\,h_{\bar s}(\beta'(t))-\beta'(t)\bigr]\,dt+o(1).
\]
Here $h_{\bar s}$ is the piecewise-linear interpolation of $u\mapsto u^2$ on the grid 
$\{0,1/\bar s,2/\bar s,\dots\}$, which has the closed form
$h_{\bar s}(u)=u^2+\bar s^{-2}\,\{\bar s u\}\bigl(1-\{\bar s u\}\bigr)$, $\{\bar s u\}:=\bar s u-\lfloor \bar s u\rfloor\in[0,1)$.

The next proposition shows that, under these assumptions, the limiting tau-leaping factorization cost is at least that of the deterministic planner. The limiting gap approaches $\tfrac12 D_\infty$ as $\bar s$ subsequently grows. 
For the linear schedule with $K=N$, the effect of random block sizes is particularly clear: the deterministic planner reveals exactly one coordinate per step and has zero factorization error, whereas the tau-leaping factorization error converges to
$\tfrac12 D_\infty$.

\begin{proposition}\label{prop:comparison-disc}
Assume \eqref{eq:coefficient-convergence}, $D_\infty>0$ and $\beta$ as in Theorem~\ref{thm:asym}, with $\beta'$ having finitely many local extrema. Then, as $N,K\to\infty$ with $N/K\to\bar s\in[1,\infty)$, the gap
$\Delta_{\bar s}:=\lim_{N,K\to\infty}\bigl(\efact^{N,K}(\beta)-\efact^{\det,N,K}(\beta)\bigr)$
exists and equals
\[
\Delta_{\bar s}
=\frac{1}{2}\int_0^1 g(\beta(t))\left[\beta'(t)-\frac{\theta(t)\bigl(1-\theta(t)\bigr)}{\bar s}\right]dt,
\qquad \theta(t):=\{\bar s\,\beta'(t)\} .
\]
Moreover $0\le\Delta_{\bar s}\le\tfrac12 D_\infty$ and
\[
\left|\Delta_{\bar s}-\tfrac12 D_\infty\right|
\ \le\
\frac{1}{8\bar s}\int_0^1 g(\beta(t))\,dt
\ \le\
\frac{\|g\|_\infty}{8\bar s}\ \xrightarrow[\bar s\to\infty]{}\ 0 .
\]
\end{proposition}

\section{Examples}\label{sec:examples}

In this section we compute the dependence density $\rho$ for various classes of target distributions $\pi\in \mathcal{P}(\x^N)$. These examples show how the shape of $\rho$ reflects the way dependence is distributed across the coordinates, and how it affects the factorization error. 
We include numerical simulations verifying the theoretical results for the stationary Markov chains and for exchangeable models. All proofs are deferred to Appendix \ref{app:ex proofs}.

The first example is that of product measures.
Since the coordinates are independent, conditioning on the revealed coordinates does not change the conditional law of unrevealed ones, so parallelizing the updates induces no factorization error. 

\begin{proposition}[Product measures]\label{prop:product-example}
Let $\pi=\pi_1\otimes\cdots\otimes\pi_N$.
Then $\rho\equiv0$ on $[0,1]$ and consequently $\varepsilon_{\mathrm{fact}}(\beta)=0$ for every schedule $\beta$.
\end{proposition}
This is the degenerate case of Lemma~\ref{lem:product-measure}; every schedule is optimal, and schedule design is vacuous.

Next we consider distributions with local dependence, starting with a stationary Markov chain. Here $\rho$ can be expressed through the mutual informations $I(X^0;X^d)$, capturing how the Markovian dependence changes as the fraction of revealed coordinates varies.

\begin{proposition}[Stationary Markov chains]\label{prop:markov-example}
Let $(X^t)_{t\in\mathbb Z}$ be a stationary time-homogeneous Markov chain on $\x$ with transition matrix $P$ and stationary law $\mu$, and let $\pi_N$ be the law of its restriction $(X^1,\dots,X^N)$, so that
\[
\pi_N(x^1,\dots,x^N)
=
\mu(x^1)\prod_{t=1}^{N-1}P(x^t,x^{t+1}).
\]
Set $h_0:=H(X^0)$, $h_d:=H(X^d\mid X^0)$ and $I_d:=I(X^0;X^d)=h_0-h_d$.
Then 
\begin{equation}\label{eq:markov finite}
\rho_N(u)=\frac{d^2}{du^2}\left[\frac{1}{N}G_N(u)\right], \qquad G_N(u):=\sum_{d=1}^{N-1}
(N-d) I_d\,u^2(1-u)^{d-1}.
\end{equation}
Moreover, if $\sum_{d\ge1}I_d<\infty$, then $\rho_N\to g$
uniformly on $[0,1]$, where
\[
g(u):=\frac{d^2}{du^2}\left[u\,\mathbb E I_{D_u}\right],
\qquad D_u\sim\mathrm{Geom}(u),\qquad u\in(0,1],
\]
and $g$ is extended continuously to $u=0$.
\end{proposition}

\begin{remark}
This case has also been studied by \cite{luxembourg2025plan}.
Their analysis shows that the factorization error incurred by parallel sampling can be controlled by updating coordinates that are sufficiently well separated.
This motivates their dilated unmasking scheme, which uses a logarithmic number of denoising iterations per block. 
\end{remark}

For irreducible and aperiodic finite-state Markov chains, the lag mutual information is summable, so Proposition \ref{prop:markov-example} yields a continuous limiting profile and Theorem \ref{thm:asym} applies. For a non-i.i.d. chain this profile is not identically zero, and every fixed smooth schedule has factorization error of order $N/K$. 
If, in addition, the limiting profile is strictly positive on $[0,1]$, Corollary \ref{cor:asymptotic-optimal-solution} and the subsequent optimal-to-linear comparison of Remark \ref{remark:optVSlin} apply. 

The next result identifies the limiting profile for a stationary process satisfying the uniform convergence assumption, of which the Markov chain is one instance. 
Stationarity relaxes the conditional independence structure of Markov chains while retaining translation invariance, a regime of interest for sequence models whose dependence is far from a Markovian one.

\begin{proposition}[Stationary processes]\label{prop:stationary-example}
Let $(X^t)_{t\in\mathbb Z}$ be a stationary process on $\x$, that is, for every finite set of times $t_1,\dots,t_k$ and every shift $s\in\mathbb{Z}$,
\[
(X^{t_1},\dots,X^{t_k})\overset{d}{=}(X^{t_1+s},\dots,X^{t_k+s}),
\]
and let $\pi_N$ be the law of $(X^1,\dots,X^N)$. 
For a finite $A\subseteq \mathbb N$ write $-A:=\{-a:a\in A\}$ and
\[
h_0:=H(X^0), \qquad h(A):=H(X^0\mid X^{-A}), \qquad 
I(A):=I(X^0;X^{-A})=h_0-h(A),
\]
with $I(\varnothing)=0$. 
Let $\xi_i\overset{iid}{\sim}\mathrm{Bern}(u)$ for 
$i\ge1 $ and 
$S_{d,u}:=\{i\in[d]:\xi_i=1\}$.
Then, for every $N$,
\[
\rho_N(u)
=
\frac{d^2}{du^2}
\left[
\frac{u}{N}
\sum_{d=0}^{N-1}
\mathbb E I(S_{d,u})
\right].
\]
If moreover Assumption~\ref{ass:G_N-asymp} holds, then, with $S_u:=\{i\ge1:\xi_i=1\}$,
\begin{equation}
    \rho_N(u)
\underset{N\to\infty}{\longrightarrow}
\frac{d^2}{du^2}
\left[
u\,\mathbb E I(S_u)
\right].
\end{equation}

\end{proposition}
For a Markov chain, $I(S_{d,u})=I_{\min S_{d,u}}$ depends on the revealed set only through its nearest element, and the $d$-sum telescopes to $\tfrac1N G_N(u)$, recovering \eqref{eq:markov finite}.

The last class of examples we consider is that of exchangeable mixtures. 
In these mixtures the dependence is mediated by a global latent variable $P$, in contrast to the local dependence of the Markov examples above. As a result, $\rho_N$ may collapse away from the origin as $N\to\infty$, while concentrating at it. 
Intuitively, once a positive proportion of coordinates has been revealed, $P$  is almost fully identified and the remaining coordinates become nearly conditionally independent.

\begin{proposition}[Exchangeable mixtures]\label{prop:exchangeable-example}
Let $\pi_N$ be an exchangeable law of the form
$P\sim Q$, and $X^1,\dots,X^N\mid P\overset{\mathrm{iid}}{\sim}P$,
so that
\[
\pi_N(x^1,\dots,x^N)
=
\int \prod_{i=1}^N p(x^i)\,Q(dp).
\]
Let
$\iota_N(i):=I(X^{i+1};X^{i+2}\mid X^{1:i})$ for $i=0,\ldots,N-2$,
then
\begin{equation}\label{eq:exchangeable finite}
\rho_N(u)
=
(N-1)\,
\mathbb E_B[\iota_N(B)],
\qquad
B\sim\mathrm{Bin}(N-2,u),
\end{equation}
If moreover $I(X^{m+1};X^{m+2}\mid X^{1:m})=o(1/m)$ and $I(X^1;X^2)>0$, then
\begin{equation}
\rho_N(u)\underset{N\to\infty}{\longrightarrow} 0 \quad\text{for } u\in(0,1],\qquad \rho_N(0)\rightarrow +\infty.    
\end{equation}
\end{proposition}

\begin{remark}
The assumption $I(X^{m+1};X^{m+2}\mid X^{1:m})=o(1/m)$ holds, for example, in the Dirichlet-categorical or Beta-Bernoulli model, in which the conditional mutual information is of order $O(1/m^2)$.  
\end{remark}

Equation \eqref{eq:exchangeable finite} is the exchangeable analogue of the Markov formula \eqref{eq:markov finite}. In the Markov case the profile averages dependence across random gap lengths, whereas in the exchangeable case it averages the residual dependence between two future coordinates after conditioning on a random number of previously revealed samples.

We close by comparing two schedules for a finite-dimensional exchangeable target using the same number of iterations. In this degenerate regime, the choice of schedule affects not only the constant in the factorization error but also its asymptotic order.

\begin{proposition}[Comparison of different schedules]\label{prop:comparison-schedules}
Suppose $Q$ is a nondegenerate $p$-dimensional parametric prior satisfying the regularity conditions of \citet{clarke1994jeffreys}, and let $K_N:=1+\lceil \log_2 N\rceil$.
Then the uniform schedule
\[
\beta_k^{\mathrm{unif}}:=\frac{k}{K_N},
\qquad k=0,\dots,K_N,
\]
satisfies
\[
\varepsilon_{\mathrm{fact}}^{\mathrm{unif}}(\pi_N,K_N)
\ge  \frac{C N}{\log N}
\]
 for a constant $C>0$, whereas the geometric schedule
\[
\beta_1=\frac1N,
\qquad
\beta_k=\min\{1,2\beta_{k-1}\},
\]
satisfies
\[
\varepsilon_{\mathrm{fact}}^{\star}(\pi_N,K_N)
\le
\varepsilon_{\mathrm{fact}}^{\mathrm{geom}}(\pi_N,K_N)
=
O(\log N).
\]
where $\varepsilon^\star_{\mathrm{fact}}(\pi_N,K_N)$ denotes the minimum of $\varepsilon_{\mathrm{fact}}$ over all schedules with $K_N$ steps. Consequently
\[
\frac{
\varepsilon_{\mathrm{fact}}^{\mathrm{unif}}(\pi_N,K_N)
}{
\varepsilon_{\mathrm{fact}}^{\star}(\pi_N,K_N)
}
\ge \frac{N C'}{(\log N)^2},
\]
for a constant $C'>0$.
\end{proposition}

\subsection{Numerical illustration}
\label{sec:numerics}

We illustrate the theory for a stationary Markov chain (local dependence, Proposition~\ref{prop:markov-example}) and for a Beta-Bernoulli exchangeable model (global dependence, Proposition~\ref{prop:exchangeable-example}).
The quantities are evaluated numerically from exact expressions for $\rho_N$, without Monte Carlo sampling.

\paragraph{Setup.}
For the \emph{Markov chain} we take a reversible and aperiodic 10-state lazy random walk on a sparse connected randomly weighted graph, so $I_d = I(X^0;X^d)$ decays geometrically in $d$.
For the \emph{exchangeable model} we take the Beta$(1,1)$-Bernoulli mixture, for which
\[
\iota(m):=I(X^{m+1};X^{m+2}\mid X^{1:m})
=\sum_{s=0}^m\PP(S_m=s)\,
I(X^{m+1};X^{m+2}\mid S_m=s),
\qquad S_m:=X^1+\dots+X^m,
\]
is computed by averaging the posterior-predictive mutual information over the Beta-Binomial sufficient statistic $S_m$. Here $\iota(m)=\iota_N(m)$ for every $N\ge m+2$.

\paragraph{Dependence density.}
Figure~\ref{fig:factorization}(a) shows $\rho_N$ for $N=64,\dots,512$ together with its limiting profile $g$ for the Markov chain. The curves are close, with small finite-$N$ corrections over the displayed range $u\in[0,0.15]$.
Panel~(b) shows $\rho_N$ for the exchangeable model over the same range of $N$: the mass concentrates near $u=0$ as $N$ grows, consistent with the degeneration of $\rho_N$ in Proposition~\ref{prop:exchangeable-example}.

\paragraph{Optimal versus linear schedule.}
For each target we compute the $\efact$-optimal schedule from Corollary~\ref{cor:recursive}, solving the recursion \eqref{eq:recursive formula} by one-dimensional root-finding on $\beta_1$ (Brent's method, \citealp{brent2013algorithms}).
Neither profile is log-concave, so the sufficient condition of Corollary~\ref{cor:recursive} does not apply. Nonetheless the numerical shooting search found a single crossing of the terminal condition $\beta_K=1$ in each case, providing numerical evidence of uniqueness for these examples. 
Panel~(c) compares the optimal and linear schedules at $N=256$, $K=\lceil\log_2 N\rceil+1=9$.
For the Markov chain the optimal schedule departs only mildly from the linear one; for the exchangeable model it is strongly front-loaded, concentrating budget where $\rho_N$ is largest.

\paragraph{Local-versus-global dependence.}
Panel~(d) plots the improvement ratio $\efact^{\mathrm{lin}}/\efact^{\mathrm{opt}}$ against $N$, with the step budget $K_N=\lceil\log_2 N\rceil+1$ growing logarithmically.
For the Markov chain the gain is modest and the ratio appears to stabilize over the tested dimension at $1.1$, consistent with the regime of  Remark~\ref{remark:optVSlin}.
For the exchangeable model the numerical results appear consistent with the growth rate lower bound $N/(\log N)^2$, predicted by Proposition~\ref{prop:comparison-schedules}: here optimizing the schedule changes the asymptotic rate of $\efact$, not merely its constant.

\begin{figure}[t]
  \centering
  \includegraphics[width=\linewidth]{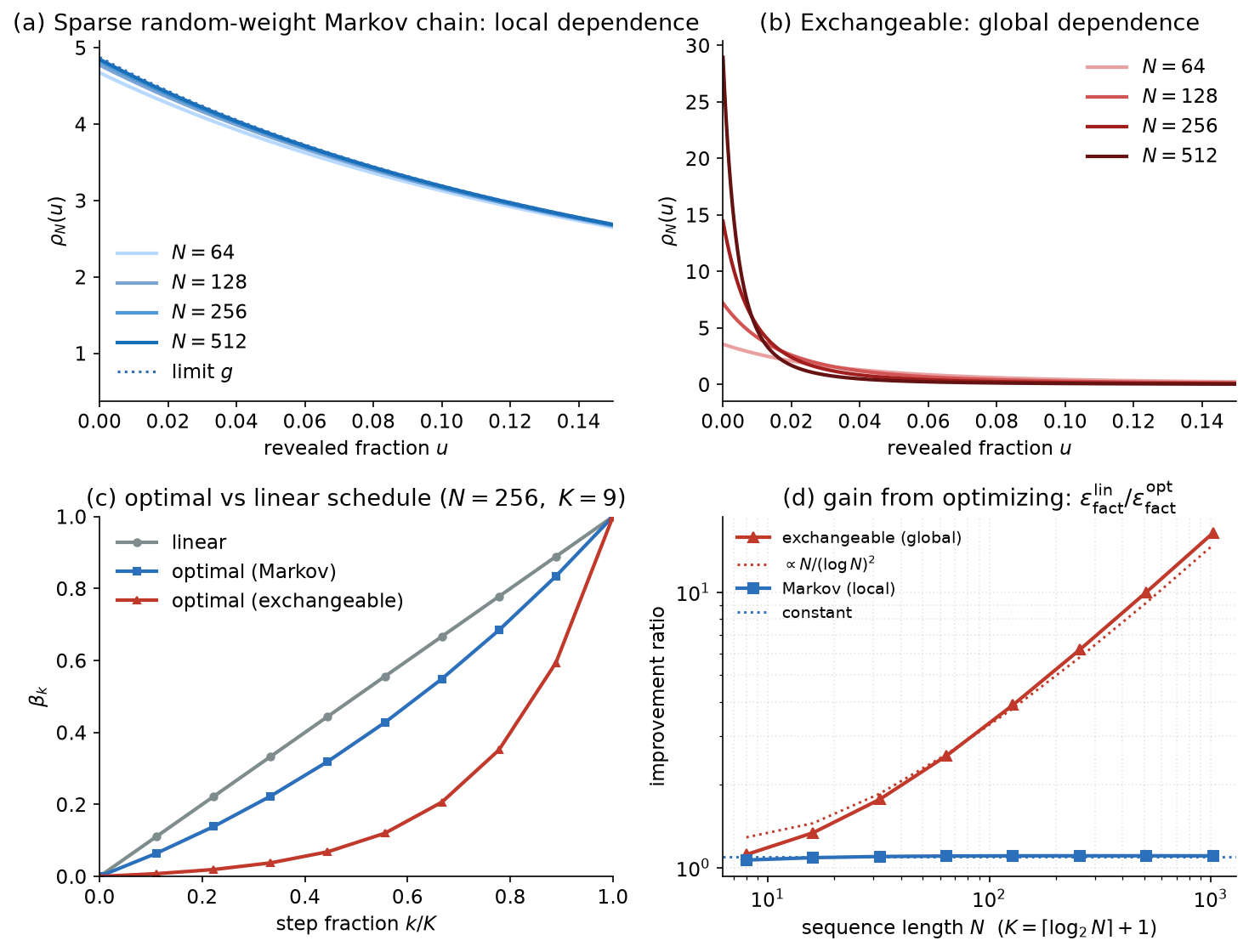}
  \caption{
    \textbf{(a)}~Dependence density $\rho_N$ for the Markov chain at $N=64,128,256,512$, with the limiting profile $g$. 
    \textbf{(b)}~$\rho_N$ for the Beta$(1,1)$-Bernoulli exchangeable model at
    $N=64,128,256,512$.
    Panels (a)-(b) display $u\in[0,0.15]$.
    \textbf{(c)}~$\efact$-optimal versus linear schedule at $N=256$, $K=9$.
    \textbf{(d)}~Improvement ratio $\efact^{\mathrm{lin}}/\efact^{\mathrm{opt}}$
    versus $N$ with $K=\lceil\log_2 N\rceil+1$.
  }
  \label{fig:factorization}
\end{figure}

\section{Estimators for the information profile}\label{sec:profile-estimators}
Computing the recursive schedule in Corollary \ref{cor:recursive} requires $\rho$, which is determined by the sequence $\iota(0),\ldots,\iota(N-2)$.
Assuming that each model call returns predictive distributions for all unrevealed coordinates, estimating the full sequence with $n$ Monte Carlo replications per entry uses $O(nN)$ model calls. 
This computation is performed offline, and its cost is amortized over subsequent generated samples, whose cost is $O(K)$ model calls.

We estimate these coefficients through the auxiliary information profile introduced by \citet{lavenant2025error} and \citet{chen2025optimal}. 
For a uniform random permutation $\sigma$ of $[N]$, define
\begin{equation}\label{eq:inf_prof}
f(i):=-\E_\sigma\!\left[H_\pi(X^{\sigma_{i+1}}\mid X^{\sigma_{\le i}})\right],
\qquad i=0,\ldots,N-1.
\end{equation}
The conditional mutual information identity and the symmetry of the uniform permutation give, for $i=0,\ldots,N-2$
\begin{equation}\label{eq:cmi-increment}
\begin{aligned}
    \iota(i) &=\E_\sigma\!\left[
H_\pi(X^{\sigma_{i+2}}\mid X^{\sigma_{\le i}})
-H_\pi(X^{\sigma_{i+2}}\mid X^{\sigma_{\le i+1}})
\right]\\
&=f(i+1)-f(i)=\Delta f(i).
\end{aligned}
\end{equation}
Thus estimates of $f$ yield estimates of $\iota$ and, in turn, of the dependence density $\rho$.

We first compare two estimators of $f$, then quantify how errors in the estimated coefficients affect the factorization error and the selected schedule. 
Fix $i\in\{0,\dots,N-1\}$. Let $X\sim\pi$ and, independently,  let $Z\subseteq[N]$ be uniform among the subsets of $[N]$ of cardinality $i$ and, conditionally on $Z$, let $J$ be uniform on $[N]\setminus Z$. By \eqref{eq:inf_prof}, $f(i)$ can be equivalently written as
\[
f(i)=\mathbb E_{X,\,Z,\,J}\big[\log\pi(X^J\mid X^Z)\big].
\]
We compare two estimators of $f(i)$ built from the model predictive distribution $p_\theta^j$. The \emph{log-probability estimator} is
\begin{equation}\label{eq:logloss-estimator}
\tilde f_i^\theta
:=
\frac{1}{N-i}\sum_{j\notin Z}\log p_\theta^j(X^j\mid X^Z),
\end{equation}
and the \emph{entropy-estimator} is
\begin{equation}\label{eq:entropy-estimator}
\hat f_i^\theta
:=
-\frac{1}{N-i}\sum_{j\notin Z} H\big(p_\theta^j(\cdot\mid X^Z)\big).
\end{equation}
We write $\tilde f_i,\hat f_i$ for the corresponding quantities under perfect learning, $p_\theta^j=\pi_j$.

\begin{proposition}\label{prop:unbiased estimators}
Under perfect learning, the estimators \eqref{eq:logloss-estimator} and \eqref{eq:entropy-estimator} are unbiased
\[
\E[\tilde f_i]=\E[\hat f_i]=f(i),
\]
and 
\[
\var(\hat f_i)\le\var(\tilde f_i).
\]
\end{proposition}

We now drop the assumption of perfect learning and relate the bias of the two estimators to the model training objective. 
Up to an additive constant independent of $\theta$, the denoising cross-entropy loss decomposes as a weighted sum of conditional KL errors \citep{ou2024your,kim2025train} defined for $i=0,\dots,N-1$ as
\[
\mathcal L_i:=\mathbb E_{Z,X^Z}\Bigl[\frac{1}{N-i}\sum_{j\notin Z}\kl{\pi_j(\cdot\mid X^Z)}{p_\theta^j(\cdot\mid X^Z)}\Bigr]
=\mathbb E_{Z,J,X^Z}\bigl[\kl{\pi_J(\cdot\mid X^Z)}{p_\theta^J(\cdot\mid X^Z)}\bigr],
\]
where $Z$ and $J$ are distributed as above.

\begin{proposition}\label{prop:bias}
For general learned predictors, the bias of the log-probability estimator is 
\[
\mathbb E\big[\tilde f_i^\theta\big]-f(i)=-\,\mathcal L_i\;\le 0.
\]
If moreover $\mathcal L_i\le 2\bigl(1-1/|\x|\bigr)^2$, then the bias of the entropy estimator satisfies
\[
\bigl|\mathbb E[\hat f_i^\theta]-f(i)\bigr|\ \le\ \phi\bigl(\sqrt{\mathcal L_i/2}\bigr),
\qquad
\phi(T):=T\log(|\x|-1)+H_2(T),
\]
where $H_2(T):=-T\log T-(1-T)\log(1-T)$ is the
binary entropy. In particular 
\[
\bigl|\E[\hat f_i^\theta]-f(i)\bigr|
=
O\!\left(
\sqrt{\mathcal L_i}\log(1/\mathcal L_i)
\right)
\qquad\text{as }\mathcal L_i\to0.
\]
\end{proposition}

Proposition \ref{prop:bias} shows that both biases vanish as $\mathcal L_i\to0$, but provides different guarantees.
The bias of the log-probability estimator is exactly
$-\mathcal L_i$ and is therefore always non-positive. For the entropy estimator, we obtain the upper bound
$O(\sqrt{\mathcal L_i}\log(1/\mathcal L_i))$ on the absolute bias; since $x=o(\sqrt x\log(1/x))$ as $x\to0^+$ the latter guarantee is weaker in its dependence on $\mathcal L_i$.
This does not imply that the entropy estimator has a larger absolute bias as the bound is obtained through Pinsker's inequality and need not be tight.

Under perfect learning, both estimators are unbiased, and Proposition~\ref{prop:unbiased estimators} shows that the entropy estimator has smaller variance than the log-probability estimator. The following provides a Monte Carlo bound for $\hat f^\theta_i$ for arbitrary predictors, including imperfectly learned ones.

\begin{proposition}\label{prop:estimator variance}
Fix $i\in\{0,\dots,N-1\}$ and let $(X^{(m)},Z^{(m)})$, $m=1,\dots,n$ be independent copies of $(X,Z)$ as defined above, and let $\hat f_{i,m}^{\theta}$ be the corresponding estimators  \eqref{eq:entropy-estimator}, and set $\bar f_i:=\tfrac1n\sum_{m}\hat f^\theta_{i,m}$. Then
\[
\var(\bar f_i)\le\frac{(\log|\x|)^2}{4n}.
\]
Moreover, for every $\delta\in(0,1)$, with probability at least $1-\delta$,
\[
\bigl|\bar f_i-\E[\hat f^\theta_i]\bigr|\le\log|\x|\,\sqrt{\frac{\log(2/\delta)}{2n}}.
\]
\end{proposition}

Given an estimated vector $\bar f$, we set
\[
\bar\iota(i):=\bar f_{i+1}-\bar f_i,
\qquad
\bar\rho(u):=(N-1)\sum_{i=0}^{N-2}\bar\iota(i)B_i^{N-2}(u).
\]

We conclude by quantifying how errors in the estimated coefficients affect the factorization error and the selected schedule. We do not claim \eqref{eq:first-ineq-efact} to be tight, and leave the development of sharper inequalities with stronger assumptions to future work.

\begin{proposition}\label{prop:stability}
Let $\mathcal B_K
:=
\left\{
\beta=(\beta_0,\ldots,\beta_K):
0=\beta_0\le\cdots\le\beta_K=1
\right\}$,
and $\Delta\beta_{\max}
:=
\max_{1\le k\le K}(\beta_k-\beta_{k-1})$.
For any vector $a=(a(0),\ldots,a(N-2))\in\mathbb R^{N-1}$, set
\[
\rho_a(u):=(N-1)\sum_{i=0}^{N-2}a(i)B_i^{N-2}(u),
\qquad
\efact(a,\beta):=N\sum_{k=1}^K\int_{\beta_{k-1}}^{\beta_k}
(\beta_k-u)\rho_a(u)\,du.
\]
For a fixed schedule $\beta$, this functional is linear in $a$. Writing $\|a\|_1:=\sum_{i=0}^{N-2}|a(i)|$, we have
\begin{equation}\label{eq:first-ineq-efact}
\left|\efact(\iota,\beta)-\efact(\bar\iota,\beta)\right|
\le N\Delta\beta_{\max}\|\iota-\bar\iota\|_1,
\end{equation}
If $\hat\beta\in\arg\min_{\beta\in\mathcal B_K}\efact(\bar\iota,\beta)$ and
$\beta^*\in\arg\min_{\beta\in\mathcal B_K}\efact(\iota,\beta)$, then
\begin{equation}\label{eq:error-f-estimator}
0\le\efact(\iota,\hat\beta)-\efact(\iota,\beta^*)
\le 2N\,\max\{\Delta\hat\beta_{\max},\Delta\beta^*_{\max}\}\,\|\iota-\bar\iota\|_1.
\end{equation}
\end{proposition}

\subsection{Numerical experiments on mixtures of products}
\label{sec:experiments}

We illustrate the theoretical predictions on a class of targets for which all the relevant conditional distributions can be evaluated exactly.
This allows us to set $p_\theta^j=\pi_j$, so that the learning-error contribution vanishes.

\paragraph{Target distribution.}
We consider a mixture of $r$ product distributions over $N$ coordinates, each taking value in a vocabulary $\x$ of size $L$:
\begin{equation}\label{eq:mixture-target}
\pi(x)
=
\sum_{z=1}^r w_z\prod_{j=1}^N\mu_{z,j}(x^j),
\qquad
x\in\x^N,
\end{equation}
where $w\in\Delta^{r-1}$ and $\mu_{z,j}\in\mathcal P(\x)$.
Equivalently, a latent variable $Z\sim\mathrm{Cat}(w)$ is drawn first, after which the coordinates are conditionally independent, with $X^j\mid Z=z\sim\mu_{z,j}$.

The component marginals are sampled from the symmetric Dirichlet distribution $\mu_{z,j}\sim\mathrm{Dir}(\alpha/L,\dots,\alpha/L)$ and then held fixed throughout profile estimation and Monte Carlo evaluation. The mixture weights are uniform, $w_z=1/r$. Throughout the experiment, we set $r=5$, $L=10$, $\alpha=0.3$ and consider $N\in\{8,16,32,64,128,256\}$, $K=\log_2N$. Implementation details are in Appendix~\ref{app:conditionals}.

\paragraph{Behavior of $\rho_N$.} 
Figure~\ref{fig:mixture} $(a)$ shows that the estimated dependence density $\bar\rho_N(u)$ becomes increasingly concentrated near $u=0$ as $N$ increases.
For readability, the upper limit of the vertical axis is set to $70$.

Conditional on the sampled component marginals, the target is generally not exchangeable, since the distributions $\mu_{z,j}$ depend on the coordinate $j$.
Thus, Proposition~\ref{prop:exchangeable-example} does not apply directly.
Nevertheless, the observed behavior suggests that a similar concentration phenomenon can occur beyond the conditionally i.i.d. setting covered by that proposition.

Figure~\ref{fig:mixture} $(b)$ compares the linear schedule $\beta_k^{\mathrm{lin}}=k/K$ with a numerically optimized schedule $\hat\beta$ for $N=128$, $K=7$, where the latter is computed from the estimated profile.
Evaluating the representation in Theorem~\ref{thm:main} with this profile gives estimated factorization errors of $21.11$ and $0.66$, respectively, corresponding to an estimated improvement factor of approximately $32$.

\paragraph{Monte Carlo comparison.}
Figure~\ref{fig:mixture} $(c)$ compares two estimates of the factorization error for the linear and numerically optimized schedules.
The solid curves are obtained from Theorem~\ref{thm:main}, using $\bar\iota(i)=\bar f_{i+1}-\bar f_i$ to construct the estimated dependence density.
The boxplots summarize $100$ independent Monte Carlo estimates from \eqref{eq:factorization-error}, each based on $4000$ draws.
This comparison provides a check of the representation-based estimates against direct Monte Carlo
evaluation.

Finally, Figure~\ref{fig:mixture} $(d)$ reports the estimated improvement ratio $\widehat R_{N,K}:=\efact(\bar\iota,\beta^{\mathrm{lin}})/\efact(\bar\iota,\hat\beta)$.
For the target instances considered, the estimated gain increases with $N$, indicating substantial benefits from schedule optimization over the tested range of dimensions.

\begin{figure}[t]
\centering
\includegraphics[width=\linewidth]{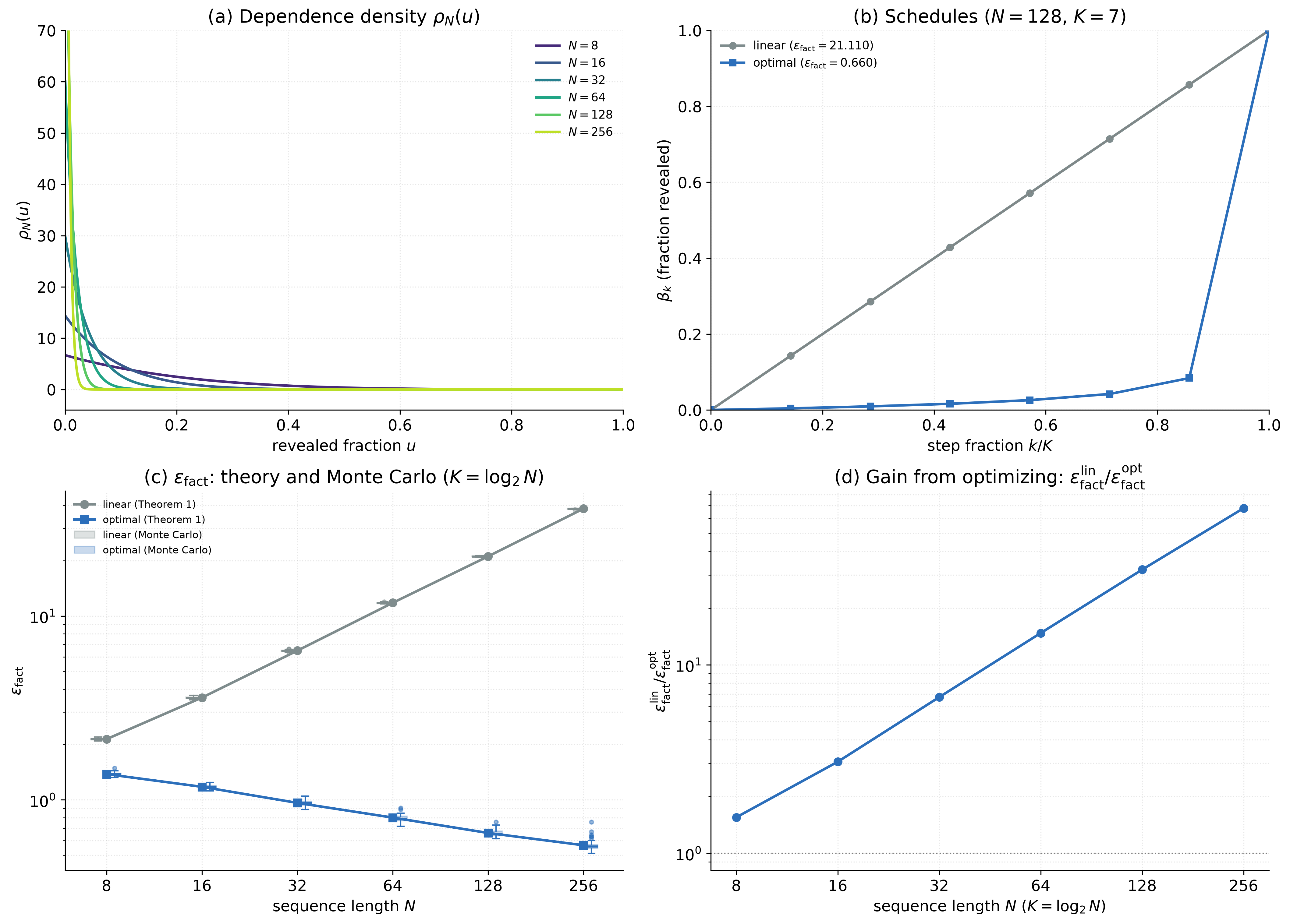}

\caption{
\textbf{(a)} Estimated dependence densities $\bar\rho_N(u)$ for mixture-of-products targets with $N\in\{8,16,32,64,128,256\}$. 
The vertical axis is truncated at $70$.
\textbf{(b)} Linear and numerically optimized schedules
for $N=128$ and $K=7$.
\textbf{(c)} Representation-based and direct Monte Carlo estimates of $\efact$ for both schedules with $K=\log_2N$. 
The curves use the estimated profiles, and the boxplots summarize $100$ independent Monte Carlo estimates.
Both axes are logarithmic. 
\textbf{(d)} Estimated improvement ratio $\widehat R_{N,K}$ versus $N$, with $K=\log_2N$, on logarithmic axes.}
\label{fig:mixture}
\end{figure}

\section*{Acknowledgements}
During the preparation of this work, the authors used OpenAI ChatGPT-5.6 to explore proof strategies and assist with coding and manuscript revision.
The authors take full responsibility for the contents of the paper.

\bibliography{references}

\appendix

\section{Proof of Section \ref{sec:background}} \label{app:proof section equivalence}
\begin{proof}[Proof of Proposition \ref{prop:tau-leaping-equivalence}]

Since both methods are Markov chains starting from the fully masked state, equality of the stepwise transition kernels implies equality of the laws of $(y_0,\dots,y_K)$, and in particular of the terminal outputs.

Fix $k\in\{1,\dots,K\}$ and recall $q_k:=\frac{\beta_k-\beta_{k-1}}{1-\beta_{k-1}}$.
Let $y\in (\x^*)^N$ be the current state with masked set  
\[
M:=\{i\in[N]: y^i=\m\},\qquad r:=|M|,
\] 
and let $x\in(\x^*)^N$ be a candidate next state. Both kernels vanish unless $x^{M^c}=y^{M^c}$; we assume this in what follows and omit the corresponding indicator from all displays. 
We write $K_k^{(1)}$ and $K_k^{(2)}$ for the step-$k$ kernels of Algorithms~\ref{alg:standard-tau-leaping} and~\ref{alg:tau-leaping2}, respectively.

In Algorithm~\ref{alg:standard-tau-leaping}, the coordinates of $M$ are updated independently, so
\[
K_k^{(1)}(x\mid y)
=
\prod_{i\in M}
\Bigl[
(1-q_k)\,\mathbbm 1_{\{x^i=\m\}}
+
q_k\,p_\theta^i(x^i\mid y)
\Bigr].
\]
For Algorithm~\ref{alg:tau-leaping2}, define the set of newly revealed coordinates
\[
R:=\{i\in M: x^i\neq \m\},\qquad u:=|R|.
\]
Since $p_\theta^i(\m\mid y)=0$, the event $\{y_k=x\}$ is equal to $\{z_k=R,\,y_k^{R}=x^{R}\}$. Therefore
\begin{align*}
K_k^{(2)}(x\mid y)
&=\PP(s_k=u)\PP(z_k=R\mid s_k=u)\PP(y_k^R=x^R\mid z_k=R)\\
&=
\binom{r}{u}q_k^{u}(1-q_k)^{r-u}\,
\frac{1}{\binom{r}{u}}
\prod_{i\in R}
p_\theta^i(x^i\mid y)
=
q_k^{u}(1-q_k)^{r-u}
\prod_{i\in R}
p_\theta^i(x^i\mid y)
\\
&=
\prod_{i\in M}
\Bigl[
(1-q_k)\,\mathbbm 1_{\{x^i=\m\}}
+
q_k\,p_\theta^i(x^i\mid y)
\Bigr]
=K_k^{(1)}(x\mid y),
\end{align*}
where the last line uses $p_\theta^i(\m\mid y)=0$ once more: the bracket equals $1-q_k$ when $x^i=\m$ and $q_k\,p_\theta^i(x^i\mid y)$ otherwise. Since $k$, $y$ and $x$ were arbitrary, the two algorithms have the same transition kernels, which proves the claim.
\end{proof}

\section{Proofs of Section \ref{sec: main result}}\label{app:section:main result}

\begin{proof}[Proof of Lemma \ref{lem:product-measure}]
The coefficient identity in Lemma~\ref{lemma:correlation} gives
\[
\tc(\pi)=\sum_{i=0}^{N-2}(N-i-1)\iota(i)
=\kl{\pi}{\otimes_{j=1}^N\pi_j}.
\]
All coefficients are positive and $\iota(i)\ge0$. If $\pi$ is a product measure, the right-hand side is zero, so every $\iota(i)$ is zero and $\rho\equiv0$. Otherwise at least one $\iota(i)$ is positive. Since every $B_i^{N-2}(u)$ is strictly positive for $u\in(0,1)$, \eqref{eq:bernstein_rep} then gives $\rho(u)>0$ throughout this interval.
\end{proof}

Before proving Theorem~\ref{thm:main}, we recall the discrete factorization-error representation of \citet{lavenant2025error} and \citet{chen2025optimal}. We state it using the conditional mutual information profile $\iota$; the relation to their information-profile formulation is given in \eqref{eq:cmi-increment}.

\begin{lemma}[Conditional mutual information representation]
\label{lem:information-increment-representation}
Let $s=(s_1,\ldots,s_K)$, where $s_k:=|z_k|$, and set $C_0:=0$, $C_k:=\sum_{\ell=1}^k s_\ell$. For $j=1,\ldots,N-1$, define
\[
r_s(j):=\min\{C_k:C_k\ge j\}.
\]
Then
\begin{equation}\label{eq:lavenant-factorization}
\varepsilon_{\mathrm{fact}}
=\E_{\nu^\beta(s)}[A(s)],
\qquad
A(s)=\sum_{i=0}^{N-2}\iota(i)\bigl(r_s(i+1)-(i+1)\bigr),
\end{equation}
where $\nu^\beta(s)$ is the induced law of the block sizes $s=(s_1,\dots,s_K):=(|z_1|,\dots,|z_K|)$.
\end{lemma}
The factor $r_s(i+1)-(i+1)$ counts the coordinates after position $i+1$ that belong to its parallel update block.

We will also use the following equivalent form of the dependence density:
\begin{equation}\label{eq:rho-beta}
\rho(u)=\sum_{i=0}^{N-2}\iota(i)b_i(u),
\qquad
b_i(u):=(N-1)B_i^{N-2}(u)
=\frac{u^i(1-u)^{N-i-2}}{B(i+1,N-i-1)},
\end{equation}
where $B(\cdot,\cdot)$ is the Beta function. In particular, each $b_i$ integrates to one.

\begin{proof}[Proof of Theorem \ref{thm:main}]
By the decomposition in Lemma~\ref{lem:information-increment-representation}, it suffices to derive an explicit expression for $\E[r_s(i)-i]$ for $i=1,\dots,N-1$, where $s=(s_1,\dots,s_K):=(|z_1|,\dots,|z_K|)$ is the vector of block sizes of $z\sim\nu^\beta$; recall that $r_s(i)$ depends on $z$ only through $s$.

Under $\nu^\beta$, the sizes satisfy
\[
s_k\mid s_1,\dots,s_{k-1}\sim \mathrm{Bin}\left(N-\sum_{j<k}s_j,\frac{\beta_k-\beta_{k-1}}{1-\beta_{k-1}}\right),
\]
and their joint law is the multinomial $\mathrm{Mult}(N;\Delta\beta_1,\dots,\Delta\beta_K)$, where $\Delta\beta_k:=\beta_k-\beta_{k-1}$.
We are going to work with the construction of the multinomial that samples $U_1,\dots,U_N\overset{\mathrm{iid}}{\sim}\mathrm{Unif}(0,1)$ and sets
\[
s_k=\#\{j:U_j\in(\beta_{k-1},\beta_k]\},\qquad C_k=\#\{j:U_j\leq \beta_k\};
\]

Let $U_{(i)}$ denote the $i$-th order statistic of $U_1,\dots,U_N$, so that $U_{(i)}\sim\mathrm{Beta}(i,N-i+1)$ by Lemma \ref{lemma:beta} in Appendix \ref{app:section:main result}, and let 
$k(i)=\min \{k:\beta_k\geq U_{(i)}\}$
be the index of the block containing position $i$.
The overshoot $r_s(i)-i$ counts the indices $j>i$ such that $U_{(j)}$ falls in the same interval $(\beta_{k(i)-1},\beta_{k(i)}]$ as $U_{(i)}$, hence
\[
r_s(i)-i=C_{k(i)}-i.
\]
Conditionally on $U_{(i)}=u$, exactly $i$ of the uniforms are at most $u$, while the remaining order statistics have the law of the ordered values of $N-i$ independent uniforms on $(u,1)$.
Therefore, writing $F(u):=\min \{\beta_k:\beta_k\geq u\}$ for the right endpoint of the interval containing $u$,
\[
C_{k(i)}-i\,|\,U_{(i)}=u\sim \mathrm{Bin}\left(N-i,\frac{F(u)-u}{1-u}\right),
\]
hence 
\begin{equation}\label{eq:overshoot-cond-exp}
    \E[r_s(i)-i\mid U_{(i)}=u]=(N-i)\frac{F(u)-u}{1-u}.
\end{equation}
Using \eqref{eq:overshoot-cond-exp}, the law of $U_{(i)}\sim \mathrm{Beta}(i,N-i+1)$, and $F(u)=\beta_k$ on $(\beta_{k-1},\beta_k]$, we obtain
\begin{align*}
    \E[r_s(i)-i]&= \frac{N-i}{B(i,N-i+1)} \sum_{k=1}^K \int_{\beta_{k-1}}^{\beta_k} \frac{\beta_k-u}{1-u}  u^{i-1}(1-u)^{N-i}du\\
    &= N \sum_{k=1}^K \int_{\beta_{k-1}}^{\beta_k} (\beta_k-u)  \frac{u^{i-1}(1-u)^{N-i-1}}{B(i,N-i)}du\\
    &=N\sum_{k=1}^K \int_{\beta_{k-1}}^{\beta_k} (\beta_k-u) b_{i-1}(u)du,
\end{align*}
where we used $(N-i)/B(i,N-i+1)=N/B(i,N-i)$ and the definition of $b_i$ in \eqref{eq:rho-beta}. 
Substitution into \eqref{eq:lavenant-factorization} yields
\begin{align}\label{eq:Nfact}
\efact(\beta)
&=N\sum_{i=0}^{N-2}\iota(i)
\sum_{k=1}^K\int_{\beta_{k-1}}^{\beta_k}(\beta_k-u)b_i(u)\,du\nonumber\\
&=N\sum_{k=1}^K\int_{\beta_{k-1}}^{\beta_k}(\beta_k-u)\rho(u)\,du,
\end{align}
by exchanging the finite sums and using \eqref{eq:rho-beta}.
\end{proof}

\begin{lemma}\label{lemma:beta}
Let $U_1,\dots,U_N\overset{iid}{\sim}\mathrm{Unif}(0,1)$, and $U_{(i)}$ denote the $i$-th order statistic. Then 
\[
U_{(i)}\sim\mathrm{Beta}(i,N-i+1).
\]
\end{lemma}
\begin{proof}[Proof of Lemma \ref{lemma:beta}]
Since $\{U_{(i)}\leq u\}$ occurs if and only if at least $i$ of the variables $U_j$ are at most $u$, and $\#\{j:U_j\leq u\}\sim \mathrm{Bin}(N,u)$, we have
\[
\PP(U_{(i)}\leq u)=\sum_{k=i}^N\binom{N}{k}u^k(1-u)^{N-k}.
\]
Differentiating with respect to $u$ and using $k\binom{N}{k}=N\binom{N-1}{k-1}$ and $(N-k)\binom{N}{k}=N\binom{N-1}{k}$,
\begin{align*}
  f_{U_{(i)}}(u)&=\sum_{k=i}^N\binom{N}{k}\Bigl(ku^{k-1}(1-u)^{N-k}-(N-k)u^k(1-u)^{N-k-1}\Bigr)\\
  &= N\sum_{j=i-1}^{N-1}\binom{N-1}{j}u^{j}(1-u)^{N-j-1}
  -N\sum_{k=i}^{N-1}\binom{N-1}{k}u^k(1-u)^{N-k-1} \\
  &=N \binom{N-1}{i-1}u^{i-1}(1-u)^{N-i},
\end{align*}
by telescoping. Since for positive integers
\[
\frac{1}{B(i,N-i+1)}=\frac{\Gamma(N+1)}{\Gamma(i)\Gamma(N-i+1)}=\frac{N!}{(i-1)!(N-i)!}=N \binom{N-1}{i-1},
\]
the final quantity is the density of the $\mathrm{Beta}(i,N-i+1)$ distribution.
\end{proof}

\begin{proof}[Proof of Corollary \ref{cor:recursive}]
If $\pi$ is a product measure, then by Proposition \ref{prop:product-example} $\efact(\beta)=0$ for every schedule $\beta$, so all of them are optimal. 

Assume now that $\pi$ is not a product measure. By Lemma \ref{lem:product-measure}, $\rho(u)>0$ for every $u\in(0,1)$.
Since $\varepsilon_\text{fact}$ is continuous in $\beta$ on the compact simplex, a minimizer exists. 
Define 
\[
\Phi(a,b)=\int_a^b(b-u)\rho(u)du,
\]
so that, by Theorem \ref{thm:main}, 
$
\varepsilon_\text{fact}(\beta_1,\dots,\beta_{K-1})=N\sum_{k=1}^K \Phi(\beta_{k-1},\beta_k)$.

We first show that every minimizer is an interior point.
For any $a<b<c$,  
\[
\Phi(a,c)-\Phi(a,b)-\Phi(b,c)=(c-b)\int_a^b \rho(u)du>0,
\]
since $\rho>0$ on $(0,1)$, hence
\[
\Phi(a,c)>\Phi(a,b)+\Phi(b,c).
\]
This shows that splitting an interval strictly decreases the cost. Consequently, if a schedule has a collapsed interval, the corresponding redundant point can be moved to split some non-trivial interval and strictly decrease the objective. Therefore the minimizer must satisfy
\[
0<\beta_1<\cdots<\beta_{K-1}<1.
\]

Next, for $a<b$,
\[
\partial_b \Phi(a,b)=\int_a^b \rho(u)du,\qquad \partial_a \Phi(a,b)=-(b-a)\rho(a).
\]
Since $\beta_k$ appears only in $\Phi(\beta_{k-1},\beta_k)$ and $\Phi(\beta_k,\beta_{k+1})$, the stationary condition at an interior minimizer gives
\[
0=\frac{1}{N}\frac{\partial\varepsilon_\text{fact} }{\partial \beta_k}= 
\int_{\beta_{k-1}}^{\beta_k}\rho(u)du-(\beta_{k+1}-\beta_k)\,\rho(\beta_k).
\]
which yields the desired recursion.

For the proof of uniqueness, the recursion formula \eqref{eq:recursive formula} can be written as
\[
\Delta_{k+1}=\Psi(\beta_k,\Delta_{k}),\quad \beta_{k}=\beta_{k-1}+\Delta_{k}, \quad \Delta_1=\beta_1,
\]
where
\[
\Psi(b,\delta):=\frac{1}{\rho(b)}\int_{b-\delta}^b \rho(u)du.
\]
First observe that $\Psi$ is strictly increasing in its second argument: for $\delta<\delta'$ with $b-\delta'\geq0$,
\[
\Psi(b,\delta')-\Psi(b,\delta)=\frac{1}{\rho(b)}\int_{b-\delta'}^{b-\delta}\rho(u)\,du>0,
\]
since $\rho>0$ on $(0,1)$.

Suppose now that two schedules with initial values $\beta_1<\tilde \beta_1$ both satisfy the recursion and the terminal condition $\beta_K=\tilde \beta_K=1$. Then $\Delta_1=\beta_1<\tilde\beta_1=\tilde\Delta_1$, and if $\beta_k<\tilde\beta_k$ and $\Delta_k<\tilde\Delta_k$, then
\[
\Delta_{k+1}=\Psi(\beta_k,\Delta_k)\leq\Psi(\tilde\beta_k,\Delta_k)<\Psi(\tilde\beta_k,\tilde\Delta_k)=\tilde\Delta_{k+1},
\]
using first the monotonicity in the first argument and then the strict monotonicity in the second; consequently $\beta_{k+1}<\tilde\beta_{k+1}$. By induction $\beta_K<\tilde\beta_K$, contradicting $\beta_K=\tilde\beta_K=1$.

It remains to verify the monotonicity in $b$ when $\ell:=\log\rho$ is concave.
Substituting $u=t+b-\delta$,
\[
\Psi(b,\delta)=\int_0^\delta e^{\ell(t+b-\delta)-\ell(b)}\,dt,
\]
so differentiating under the integral sign gives
\[
    \partial_b\Psi(b,\delta)
    =\int_0^\delta \bigl[\ell'(t+b-\delta)-\ell'(b)\bigr]\, e^{\ell(t+b-\delta)-\ell(b)}\,dt\ \geq\ 0,
\]
since concavity makes $\ell'$ nonincreasing and $t+b-\delta\leq b$. 

\end{proof}

\begin{proof}[Proof of Remark \ref{prop:heuristics}]
First let  $\rho(u)= u^{-2}$ and consider the geometric schedule $\beta_k=(1+a)\beta_{k-1}$ for $k=2,\dots,K$, with $\beta_1=(1+a)^{-(K-1)}$ so that $\beta_K=1$. For every $k\geq2$
\[
\int_{\beta_{k-1}}^{\beta_k} \frac{1}{u^2}\,du=\int_{\beta_k/(1+a)}^{\beta_k} \frac{1}{u^2}\,du = \frac{a}{\beta_k},
\]
so, since $\rho(\beta_k)= \beta_k^{-2}$, the increment prescribed by \eqref{eq:recursive formula} is
\[
\beta_{k+1}-\beta_k = \frac{1}{\rho(\beta_k)}\int_{\beta_{k-1}}^{\beta_k} \rho(u)\,du = a\beta_k,
\]
which is exactly the next geometric increment: the relations for $k=2,\dots,K-1$ hold. The relation for $k=1$ involves $\int_0^{\beta_1}u^{-2}\,du=\infty$ and therefore cannot hold; this is the only failing relation, and the sense in which the geometric schedule solves the recursion.

\end{proof}

\subsection{Proofs of subsection \ref{sec:upper bounds}}
\label{app:proof upper bounds}

\begin{proof}[Proof of Lemma \ref{lemma:correlation}]
By \eqref{eq:inf_prof}, $f(0)=-N^{-1}\sum_{j=1}^NH(X^j)$ and
$f(N-1)=-N^{-1}\sum_{j=1}^NH(X^j\mid X^{[N]\setminus\{j\}})$. Hence
\[
D(\pi)=f(N-1)-f(0)=\sum_{i=0}^{N-2}\iota(i).
\]
The entropy chain rule, averaged over a uniform permutation, gives
$H(X)=-\sum_{j=0}^{N-1}f(j)$. Therefore
\[
\tc(\pi)=\sum_{j=0}^{N-1}\bigl(f(j)-f(0)\bigr)
=\sum_{j=0}^{N-1}\sum_{i=0}^{j-1}\iota(i)
=\sum_{i=0}^{N-2}(N-i-1)\iota(i).
\]
Using $\dtc(\pi)=ND(\pi)-\tc(\pi)$ then yields
\[
\dtc(\pi)=\sum_{i=0}^{N-2}(i+1)\iota(i).
\]
For the integral identities, $b_i$ is the density of $\mathrm{Beta}(i+1,N-i-1)$, so
\[
\int_0^1b_i(u)\,du=1,\qquad
N\int_0^1u\,b_i(u)\,du=i+1,\qquad
N\int_0^1(1-u)b_i(u)\,du=N-i-1.
\]
Multiplying these identities by $\iota(i)$, summing over $i$, and applying \eqref{eq:rho-beta} proves the claimed representations. Finally, both $i+1$ and $N-i-1$ are at least one on the summation range, so nonnegativity of $\iota$ gives $D(\pi)\le\min\{\tc(\pi),\dtc(\pi)\}$.
\end{proof}

\begin{proof}[Proof of Proposition \ref{prop:uniform-bound}]
The intervals $(\beta_{k-1},\beta_k]$, $k=1,\dots,K$, partition $(0,1]$, and for any $u\in(\beta_{k-1},\beta_k]$,
\[
\beta_k-u\le \Delta\beta_k\le \Delta\beta_{\max}.
\]
Therefore, by Theorem \ref{thm:main} and Lemma \ref{lemma:correlation},
\[
\varepsilon_\text{fact}(\beta)\le N \Delta\beta_{\max}\int_0^1\rho(u)du=N\Delta\beta_{\max} D(\pi),
\]
The constant step size $\Delta\beta_k=1/K$ gives $\Delta\beta_{\max}=1/K$, from which we  deduce the second display of the statement.

Finally, $I(X^i;X^{[N]-i})\le H(X^i)\le\log|\x|$ for every $i$, since conditional entropy is nonnegative and $X^i$ takes values in $\x$; averaging over $i$ yields $D(\pi)\le\log|\x|$.
\end{proof}

\begin{proof}[Proof of Proposition \ref{prop:boundprevious-work}]
For the geometric schedule, split the first interval from the others as
\[
\varepsilon_{\mathrm{fact}}
=
N\int_0^{\beta_1}(\beta_1-u)\rho(u)\,du
+
N\sum_{k=2}^K\int_{\beta_{k-1}}^{\beta_k}(\beta_k-u)\rho(u)\,du.
\]
For the first term, since $\beta_1=a/N$,
\[
N\int_0^{\beta_1}(\beta_1-u)\rho(u)\,du
\le
N\beta_1\int_0^{\beta_1}\rho(u)\,du
\le
a\,D(\pi).
\]
For $k\ge 2$ and $u\in[\beta_{k-1},\beta_k]$, we have
\[
\beta_k-u
\le
\beta_k-\beta_{k-1}
\le
a\beta_{k-1}
\le
au,
\]
hence
\[
N\sum_{k=2}^K\int_{\beta_{k-1}}^{\beta_k}(\beta_k-u)\rho(u)\,du
\le
aN\int_{\beta_1}^1 u\,\rho(u)\,du
\le
a\,\dtc(\pi),
\]
using Lemma \ref{lemma:correlation} and $\rho\ge0$. 
Combining the two estimates gives $\varepsilon_{\mathrm{fact}}(\beta)\le a(\dtc(\pi)+D(\pi))$, and $D(\pi)\leq DTC(\pi)$ yields the stated bound.

For the reverse geometric schedule split the last interval instead. Since $1-\beta_{K-1}=a/N$ and $\beta_K=1$,
\[
N\int_{\beta_{K-1}}^{1}(1-u)\,\rho(u)\,du
\le
N(1-\beta_{K-1})\int_{\beta_{K-1}}^{1}\rho(u)\,du
=
a\int_{\beta_{K-1}}^{1}\rho(u)\,du
\le
a\,D(\pi).
\]
For $k\le K-1$ and $u\in(\beta_{k-1},\beta_k]$ we have $1-\beta_{k-1}\le(1+a)(1-\beta_k)$, hence
\[
\beta_k-u\le\beta_k-\beta_{k-1}=(1-\beta_{k-1})-(1-\beta_k)\le a(1-\beta_k)\le a(1-u),
\]
so that
\[
N\sum_{k=1}^{K-1}\int_{\beta_{k-1}}^{\beta_k}(\beta_k-u)\,\rho(u)\,du
\le
aN\int_0^{\beta_{K-1}} (1-u)\,\rho(u)\,du
\le
a\,\tc(\pi).
\]
Combining and using $D(\pi)\le \tc(\pi)$ concludes the proof.
\end{proof}

\section{Proofs of Section \ref{sec:asymptotic}}\label{proof:asymptotic}

\begin{proof}[Proof of Lemma \ref{lem:uniform-rho}]
Set $a_{N,j}:=(N-1)\iota_N(j)$ and
\[
G_N(u):=\sum_{j=0}^{N-2}g\!\left(\frac{j}{N-2}\right)B_j^{N-2}(u).
\]
Since the Bernstein basis is nonnegative and sums to one,
\[
\|\rho_N-G_N\|_\infty
\le\max_{0\le j\le N-2}\left|a_{N,j}-g\!\left(\frac{j}{N-2}\right)\right|
\longrightarrow0.
\]
Bernstein's approximation theorem gives $\|G_N-g\|_\infty\to0$ because $g$ is continuous. The triangle inequality proves the claim.
\end{proof}

\begin{proof}[Proof of Theorem \ref{thm:asym}]

Let $h=1/K$, $t_k=kh$ and $\Delta\beta_k:=\beta_k-\beta_{k-1}$ and set
\[
I_k^N=\int_{\beta_{k-1}}^{\beta_k } (\beta_k-u)\rho_N(u)du,
\]
so that  $\efact^{N,K}(\beta)=N\sum_{k=1}^K I_k^N$
by Theorem~\ref{thm:main}. 

First notice that since $\beta\in\mathcal C^1([0,1])$ we have $\max_k\Delta\beta_k\le\|\beta'\|_\infty h=O(K^{-1})$, and therefore
\[
\sum_{k=1}^K(\Delta\beta_k)^2\le\Bigl(\max_k\Delta\beta_k\Bigr)\sum_{k=1}^K\Delta\beta_k=O(K^{-1}).
\]
The proof proceeds in three steps.

\emph{Step 1: replacing $\rho_N$ by $g$.}  Let
$\eta_N:=\|\rho_N-g\|_\infty$, which tends to zero by Assumption~\ref{ass:G_N-asymp}.  Then
\begin{align*}
\left|
\sum_{k=1}^K I_k^N
-
\sum_{k=1}^K
\int_{\beta_{k-1}}^{\beta_k}
(\beta_k-u)g(u)\,du
\right|
&\le
\eta_N
\sum_{k=1}^K
\int_{\beta_{k-1}}^{\beta_k}
(\beta_k-u)\,du \\
&=
\frac{\eta_N}{2}
\sum_{k=1}^K
(\Delta\beta_k)^2=o(K^{-1}).
\end{align*}

\emph{Step 2: freezing $g$ on each interval.}
As $g$ is uniformly continuous, let $\omega_g$ denote its modulus of continuity; then for $u\in[\beta_{k-1},\beta_k]$
\[
|g(u)-g(\beta_{k-1})|\le \omega_g(\Delta\beta_k),\qquad \omega_g(\delta)\underset{\delta\downarrow 0}{\longrightarrow} 0.
\]
Hence
\[
\int_{\beta_{k-1}}^{\beta_k}
(\beta_k-u)g(u)\,du
=
\frac12 g(\beta_{k-1})(\Delta\beta_k)^2
+
O\!\left((\Delta\beta_k)^2\omega_g(\Delta\beta_k)\right).
\]
Summing over $k$ and using $\max_k\Delta\beta_k=O(K^{-1})$, gives
\[
\sum_{k=1}^K
\int_{\beta_{k-1}}^{\beta_k}
(\beta_k-u)g(u)\,du
=
\frac12
\sum_{k=1}^K
g(\beta_{k-1})(\Delta\beta_k)^2
+
o(K^{-1}).
\]

\emph{Step 3: Riemann sum.} Let $\omega_{\beta'}$ be the modulus of continuity of $\beta'$.  By the mean value theorem, $\Delta\beta_k=h\beta'(\xi_k)$ for some $\xi_k\in(t_{k-1},t_k)$, and $|\beta'(\xi_k)-\beta'(t_{k-1})|\le\omega_{\beta'}(h)$, so 
\[
\bigl|(\Delta\beta_k)^2-h^2\beta'(t_{k-1})^2\bigr|
\le
2\|\beta'\|_\infty\,\omega_{\beta'}(h)\,h^2
=
o(h^2)
\]
uniformly in $k$.
Since $g$ is bounded, summing the $K$ error terms gives $o(h)$, and therefore
\begin{align*}
\frac12
\sum_{k=1}^K
g(\beta_{k-1})(\Delta\beta_k)^2
&=
\frac{h}{2}
\left[
h\sum_{k=1}^K
g(\beta(t_{k-1}))\beta'(t_{k-1})^2
\right]
+
o(h) \\
&=
\frac{1}{2K}
\int_0^1
g(\beta(t))\beta'(t)^2\,dt
+
o(K^{-1}),
\end{align*}
the last step because $t\mapsto g(\beta(t))\beta'(t)^2$ is continuous, so the bracket is a Riemann sum converging to its integral.

Combining these estimates and multiplying by $N$ proves the claim.
Note that the error term of Step~1 requires only $N\to\infty$, while those of Steps~2 and~3 depend only on $K$; the two limits therefore decouple, and no relation between $N$ and $K$ is needed.
\end{proof}

\begin{proof}[Proof of Corollary \ref{cor:asymptotic-optimal-solution}]
Let 
\[
\mathcal I(\beta)=\int_0^1 g(\beta(t))\beta'(t)^2\,dt.
\]
To characterize the minimizer of $\mathcal I(\beta)$, observe that 
\[
\left(\int_0^1 \sqrt{g(u)}\,du\right)^2=\left(\int_0^1 \sqrt{g(\beta(t))}\,\beta'(t)\,dt\right)^2
\le
\int_0^1 g(\beta(t))\beta'(t)^2\,dt,
\]
where the first equality comes from a change of variable, since $\beta$ is increasing and $\beta(0)=0$, $\beta(1)=1$, while the second comes from Cauchy-Schwarz. As a consequence equality holds if and only if $t\mapsto\sqrt{g(\beta(t))}\,\beta'(t)$ is constant, that is, $\beta'(t)\propto g(\beta(t))^{-1/2}$.
 
It remains to check that this equality case is attained by an admissible schedule. Since $g$ is continuous and strictly positive on $[0,1]$, the function $G(y)=\int_0^y\sqrt{g(u)}\,du$ is a strictly increasing $\mathcal C^1$ bijection from $[0,1]$ onto $[0,G(1)]$ with $G'=\sqrt g>0$ on $[0,1]$, so $\beta^*(t):=G^{-1}(t\,G(1))$ is well defined, of class $\mathcal C^1$ on $[0,1]$, and satisfies
\[
\beta^*(0)=0,\qquad \beta^*(1)=1,\qquad
(\beta^*)'(t)=\frac{G(1)}{\sqrt{g\bigl(\beta^*(t)\bigr)}}>0 ,
\]
so that $\sqrt{g(\beta^*)}\,(\beta^*)'\equiv G(1)$ is constant and $\mathcal I(\beta^*)=G(1)^2$. Substituting into \eqref{eq:asympt epsilon} yields the stated expression for $\efact^{N,K}(\beta^*)$.

\end{proof}

\begin{proof}[Proof of Proposition \ref{prop:comparison-disc}]
By \eqref{eq:coefficient-convergence}, Lemma~\ref{lem:uniform-rho} gives Assumption~\ref{ass:G_N-asymp}, and $D(\pi_N)\to D_\infty$. Dividing the rescaled increments in \eqref{eq:coefficient-convergence} by $D(\pi_N)$ therefore gives uniform convergence to $g/D_\infty$, which is the normalized profile required by \citet[Assumption~1]{lavenant2025error}; changing between their grid and $j/(N-2)$ does not affect the limit, by continuity of $g$. Both limits defining $\Delta_{\bar s}$ then exist, by Theorem~\ref{thm:asym} and by \citet[Thm.~15]{lavenant2025error} respectively, after undoing their normalization. Subtracting the two limiting expressions and writing $\theta(t)=\{\bar s\beta'(t)\}$,
\[
2\Delta_{\bar s}
=
\int_0^1 g\bigl(\beta(t)\bigr)\Bigl[\bar s\,\beta'(t)^2-\bar s\,h_{\bar s}\bigl(\beta'(t)\bigr)+\beta'(t)\Bigr]dt
=
\int_0^1 g\bigl(\beta(t)\bigr)\left[\beta'(t)-\frac{\theta(t)\bigl(1-\theta(t)\bigr)}{\bar s}\right]dt,
\]
which is the stated identity.

For the sign, since $\bar s\beta'(t)=\lfloor\bar s\beta'(t)\rfloor+\theta(t)$ with $\theta(t)\in[0,1)$,
\[
\theta(t)\bigl(1-\theta(t)\bigr)\le\theta(t)\le\bar s\,\beta'(t),
\]
so the integrand is nonnegative and $\Delta_{\bar s}\ge0$.

For the size of the difference, using $0\le\theta(1-\theta)\le\tfrac14$ in the identity above,
\[
D_\infty-\frac{1}{4\bar s}\int_0^1 g\bigl(\beta(t)\bigr)dt
\ \le\
2\Delta_{\bar s}
\ \le\
D_\infty ,
\]
which gives both $0\le\Delta_{\bar s}\le\tfrac12 D_\infty$ and
$\bigl|\Delta_{\bar s}-\tfrac12 D_\infty\bigr|\le\tfrac{1}{8\bar s}\int_0^1 g(\beta(t))\,dt\le\|g\|_\infty/(8\bar s)$.
\end{proof}

\section{Proofs of Section \ref{sec:examples}}\label{app:ex proofs}
We first express the conditional mutual information profile in terms of average subset entropies, then derive a representation of $\rho$ useful for stationary targets.

\begin{lemma}\label{lemma:fH}
For $m=0,\dots,N$, define
\[
\bar H_m:=\frac1{\binom Nm}\sum_{A\subseteq[N],\,|A|=m}H(X^A).
\]
Then $f(i)=\bar H_i-\bar H_{i+1}$ for $i=0,\dots,N-1$, and $\iota(i)=2\bar H_{i+1}-\bar H_i-\bar H_{i+2}$ for $i=0,\dots,N-2$.
\end{lemma}
\begin{proof}[Proof of Lemma \ref{lemma:fH}]
The pair $(A,j):=(\sigma_{\leq i},\sigma_{i+1})$ is uniform over the ordered pairs with $|A|=i$ and $j\notin A$, so by the chain rule $H(X^{A\cup\{j\}})=H(X^A)+H(X^j\mid X^A)$,
\[
\E_\sigma\bigl[H(X^{\sigma_{i+1}}\mid X^{\sigma_{\leq i}})\bigr]
=\frac{1}{\binom{N}{i}(N-i)}\sum_{|A|=i,\ j\notin A}\bigl(H(X^{A\cup\{j\}})-H(X^A)\bigr).
\]
Each $B$ with $|B|=i+1$ arises from exactly $i+1$ such pairs $A\cup\{j\}$, and $\binom Ni(N-i)=\binom N{i+1}(i+1)$, so the first sum equals $\bar H_{i+1}$; the second equals $\bar H_i$, since each $A$ occurs for $N-i$ choices of $j$. Hence $f(i)=\bar H_i-\bar H_{i+1}$, and the expression for $\iota(i)=\Delta f(i)=f(i+1)-f(i)$ follows.
\end{proof}

\begin{lemma}\label{lemma:representation}
For $u\in[0,1]$ let $T_u:=\{t\in[N]:\xi_t=1\}$ with $\xi_1,\dots,\xi_N\overset{\mathrm{iid}}{\sim}\mathrm{Bern}(u)$, and define $\mathcal H(u):=\E\bigl[H(X^{T_u})\bigr]$ with the entropy evaluated for each fixed subset. Then
\[
\rho(u)=-\tfrac1N\,\mathcal H''(u).
\]
\end{lemma}

\begin{proof}[Proof of Lemma \ref{lemma:representation}]
Recall that $B_{m}^{N}:=\binom Nm u^m(1-u)^{N-m}$ is the $m$-th Bernstein basis polynomial of degree $N$, with the convention $B_{j}^{n}\equiv0$ for $j\notin\{0,\dots,n\}$.

Conditionally on $|T_u|=m$, the set $T_u$ is uniform among the subsets of size $m$, so
\[
\bar H_m:=\E\left[H(X^{T_u})\,\big| \, |T_u|=m\right]=\frac1{\binom Nm}\sum_{A\subseteq[N],\,|A|=m}H(X^A)
\]
and
\begin{align*}
\mathcal H(u)&=\E\left[\E\left[H(X^{T_u})\,\big| \, |T_u|\right]\right]\\
&=\sum_{m=0}^N \binom Nm u^m(1-u)^{N-m}\,\bar H_m =\sum_{m=0}^N B_{m}^{N}(u)\,\bar H_m.    
\end{align*}
Differentiating twice the Bernstein identity $\frac{d}{du}\sum_{m}B_{m}^{n}(u)c_m=n\sum_m B_{m}^{n-1}(u)(c_{m+1}-c_m)$ and re-indexing,
\[
\mathcal H''(u)=N(N-1)\sum_{i=0}^{N-2}\bigl(\bar H_{i+2}-2\bar H_{i+1}+\bar H_i\bigr)B_i^{N-2}(u).
\]
By Lemma \ref{lemma:fH} the bracket equals $-\iota(i)$, and by \eqref{eq:bernstein_rep} we obtain
\[
\mathcal H''(u)
=-N(N-1)\sum_{i=0}^{N-2}\iota(i)B_i^{N-2}(u)
=-N\rho(u). 
\]
\end{proof}

\begin{proof}[Proof of Proposition \ref{prop:stationary-example}]
Write $\mathcal H_N(u):=\E[H(X^{T_u})]$ with $T_u\subseteq[N]$ as in Lemma~\ref{lemma:representation}. 
For each realization of $T_u$, the chain rule in increasing order of the indices gives 
\[
H(X^{T_u})=\sum_{t=1}^N\xi_t\,H(X^t\mid X^{T_u\cap[t-1]}).
\]
Since $\xi_t$ is independent of $T_u\cap[t-1]$,
\[
\mathcal H_N(u)=u\sum_{t=1}^N\E\bigl[H(X^t\mid X^{T_u\cap[t-1]})\bigr].
\]
The lag set $\{t-a:a\in T_u\cap[t-1]\}$ has the law of $S_{t-1,u}$, and by stationarity $H(X^t\mid X^{T_u\cap[t-1]})=h(S_{t-1,u})$ in distribution. Using $h(A)=h_0-I(A)$ and $I(\varnothing)=0$,
\begin{equation}\label{eq:proof invariant}
\mathcal H_N(u)=u\sum_{t=1}^{N}\E\bigl[h(S_{t-1,u})\bigr]=N h_0 u-u\sum_{t=1}^{N-1}\E\bigl[I(S_{t,u})\bigr].
\end{equation}
Since the first term is linear in $u$, Lemma~\ref{lemma:representation} gives the finite-$N$ formula.

For the limit, realize all the Bernoulli sets on one probability space: let $\xi_d\overset{\mathrm{iid}}{\sim}\mathrm{Bern}(u)$ for $d\ge1$ and put $S_{t,u}:=\{d\in[t]:\xi_d=1\}$, $S_u:=\{d\ge1:\xi_d=1\}$, so that $S_{t,u}=S_u\cap[t]\uparrow S_u$. For $A\subseteq\mathbb N$ set $I(A):=\lim_{m}I(A\cap[m])$, the limit existing because $m\mapsto I(A\cap[m])$ is nondecreasing and bounded by $h_0$. Then $I(S_{t,u})\uparrow I(S_u)$ surely, so $\E I(S_{t,u})\to\E I(S_u)$ by monotone convergence, and by Cesàro
\[
\Phi_N(u):=\frac{u}{N}\sum_{t=1}^{N-1}\E\bigl[I(S_{t,u})\bigr]\ \longrightarrow\ \Phi_\infty(u):=u\,\E\bigl[I(S_u)\bigr]
\qquad\text{for every }u\in[0,1].
\]
By the above, $\rho_N=\Phi_N''$, and under Assumption~\ref{ass:G_N-asymp} $\rho_N\to g$ uniformly. Since $\Phi_N(0)=0$, Taylor's formula with integral remainder gives
\[
\Phi_N(u)=\Phi_N'(0)\,u+\int_0^u (u-s)\,\Phi_N''(s)\,ds .
\]
Taking $u=1$ and using the uniform convergence of $\Phi_N''$ together with $\Phi_N(1)\to\Phi_\infty(1)$, the sequence $\Phi_N'(0)$ converges, say to $c$. Passing to the limit in the display yields $\Phi_\infty(u)=cu+\int_0^u(u-s)g(s)\,ds$, so $\Phi_\infty$ is twice continuously differentiable with $\Phi_\infty''=g$, that is,
\[
\rho_N(u)\longrightarrow\frac{d^2}{du^2}\bigl[u\,\E I(S_u)\bigr]. \qedhere
\]
\end{proof}

\begin{proof}[Proof of Proposition \ref{prop:markov-example}]
By the Markov property, conditioning on the whole revealed past is equivalent to conditioning on the nearest revealed coordinate, so $I(A)=I_{\min A}$ for every nonempty finite $A\subseteq\mathbb N$, and $I(\varnothing)=0$. Since $\PP(\min S_{t,u}=j)=u(1-u)^{j-1}$ for $j=1,\dots,t$,
\[
\E\bigl[I(S_{t,u})\bigr]=\sum_{j=1}^{t}u(1-u)^{j-1}I_j .
\]
Substituting into \eqref{eq:proof invariant} and exchanging the order of summation, each $j\in\{1,\dots,N-1\}$ occurring for the $N-j$ indices $t\in\{j,\dots,N-1\}$,
\[
u\sum_{t=1}^{N-1}\E\bigl[I(S_{t,u})\bigr]=\sum_{j=1}^{N-1}(N-j)\,I_j\,u^2(1-u)^{j-1}=G_N(u),
\]
so that $\mathcal H_N(u)=Nh_0u-G_N(u)$, and \eqref{eq:markov finite} follows from Lemma~\ref{lemma:representation}.

For the limit, set $w_j(u):=u^2(1-u)^{j-1}$, so that $\rho_N=\sum_{j=1}^{N-1}\bigl(1-\tfrac jN\bigr)I_j\,w_j''$. Writing $m=j-1$,
\[
w_j''(u)=2(1-u)^{m}-4mu(1-u)^{m-1}+m(m-1)u^2(1-u)^{m-2},
\]
and since $\max_u u(1-u)^{m-1}\le 1/m$ and $\max_u u^2(1-u)^{m-2}\le 4/m^2$, we get $\sup_u|w_j''(u)|\le10$ for every $j$. Therefore, with $\rho_\infty:=\sum_{j\ge1}I_j w_j''$,
\[
\bigl\|\rho_N-\rho_\infty\bigr\|_\infty
\le
\frac{10}{N}\sum_{j=1}^{N-1}j\,I_j+10\sum_{j\ge N}I_j .
\]
Both terms vanish as $N\to\infty$: the second because $\sum_j I_j<\infty$, the first by Kronecker's lemma. The convergence is therefore uniform, and since the differentiated series converges uniformly, term-by-term differentiation of $\sum_{j\ge1}I_jw_j(u)=u\,\E[I_{D_u}]$, $D_u\sim\mathrm{Geom}(u)$, is legitimate, giving $\rho_\infty=\frac{d^2}{du^2}\bigl(u\,\E[I_{D_u}]\bigr)$.
\end{proof}

\begin{proof}[Proof of Proposition \ref{prop:exchangeable-example}]
For every fixed permutation, exchangeability gives
\[
I(X^{\sigma_{i+1}};X^{\sigma_{i+2}}\mid X^{\sigma_{\le i}})
=I(X^{i+1};X^{i+2}\mid X^{1:i}),
\]
so that $\iota_N(m)=\iota(m)$ for $m\le N-2$. Substitution into \eqref{eq:rho-binomial} proves \eqref{eq:exchangeable finite}.

For the asymptotics, fix $u\in(0,1)$, let $B_N\sim\mathrm{Bin}(N-2,u)$ and $A_N:=\{B_N\ge\tfrac u2(N-2)\}$. By a Chernoff bound there is $c_u>0$, nondecreasing in $u$, with $\PP(A_N^c)\le e^{-c_uN}$. Since $\x$ is finite, $0\le\iota(m)\le\log|\x|$, so $(N-1)\E[\iota(B_N)\mathbf 1_{A_N^c}]\le(N-1)\log|\x|\,e^{-c_uN}\to0$. Given $\varepsilon>0$, choose $m_\varepsilon$ with $\iota(m)\le\varepsilon/m$ for $m\ge m_\varepsilon$; for $N$ large enough, $B_N\ge\tfrac u2(N-2)\ge m_\varepsilon$ on $A_N$, so
\[
(N-1)\E\bigl[\iota(B_N)\mathbf 1_{A_N}\bigr]\le\varepsilon\,\frac{2(N-1)}{u(N-2)} .
\]
Hence $\limsup_N\rho_N(u)\le2\varepsilon/u$ for every $\varepsilon>0$, so $\rho_N(u)\to0$. 
For $u=1$, $B_N=N-2$ almost surely and $\rho_N(1)=(N-1)\iota(N-2)\to0$ since $\iota(m)=o(1/m)$. For $u=0$, $B_N=0$ almost surely and $\rho_N(0)=(N-1)\iota(0)=(N-1)I(X^1;X^2)\to+\infty$.
\end{proof}

\begin{proof}[Proof of Proposition \ref{prop:comparison-schedules}]
Write $\iota(m)=I(X^{m+1};X^{m+2}\mid X^{1:m})$ and $K_N=1+\lceil\log_2N\rceil$; nondegeneracy of $Q$ gives $\iota(0)=I(X^1;X^2)>0$.

\emph{Uniform schedule.} Since $\iota(m)\ge0$, keeping only the $m=0$ term in \eqref{eq:exchangeable finite} gives $\rho_N(u)\ge(N-1)\iota(0)(1-u)^{N-2}$, and keeping only $k=1$ in \eqref{eq:mainresult},
\[
\varepsilon_{\mathrm{fact}}^{\mathrm{unif}}
\ \ge\
N(N-1)\iota(0)\int_0^{1/K_N}\Bigl(\frac1{K_N}-u\Bigr)(1-u)^{N-2}\,du .
\]
For $N$ large enough, $K_N\le N/2$, so $1/N\le 1/(2K_N)$; restricting the integral to $[0,1/N]$ and using $\tfrac1{K_N}-u\ge\tfrac1{K_N}-\tfrac1N\ge\tfrac1{2K_N}$ together with $(1-u)^{N-2}\ge(1-1/N)^{N-2}\ge\tfrac13$ there,
\[
\varepsilon_{\mathrm{fact}}^{\mathrm{unif}}
\ \ge\
N(N-1)\iota(0)\cdot\frac{1}{2K_N}\cdot\frac13\cdot\frac1N
=\frac{(N-1)\iota(0)}{6K_N}
\ \ge\ \frac{CN}{\log N}
\]
for a constant $C>0$, since $K_N=O(\log N)$.

\emph{Geometric schedule.} The schedule $\beta_1=1/N$, $\beta_k=\min\{1,2\beta_{k-1}\}$ is the case $a=1$ of Proposition~\ref{prop:boundprevious-work} and reaches $1$ in at most $K_N$ steps, so
\[
\varepsilon_{\mathrm{fact}}^{\star}(\pi_N,K_N)\le\varepsilon_{\mathrm{fact}}^{\mathrm{geom}}(\pi_N,K_N)\le2\,\dtc(\pi_N).
\]
By exchangeability $H(X^i\mid X^{[N]\setminus\{i\}})=H(X^N\mid X^{1:N-1})\ge H(X^N\mid X^{1:N-1},P)=\E[H(P)]$, and $H(X^{1:N}\mid P)=N\,\E[H(P)]$ because the coordinates are conditionally i.i.d.; hence
\[
\dtc(\pi_N)=H(X^{1:N})-N\,H(X^N\mid X^{1:N-1})\le H(X^{1:N})-H(X^{1:N}\mid P)=I(P;X^{1:N}).
\]
For a $p$-dimensional parametric prior satisfying the regularity conditions of \citet{clarke1994jeffreys}, $I(P;X^{1:N})=\tfrac p2\log N+O(1)$, so $\varepsilon_{\mathrm{fact}}^{\star}(\pi_N,K_N)=O(\log N)$. Combining the two bounds gives the stated ratio.
\end{proof}

\section{Proofs of Section \ref{sec:profile-estimators}}

\begin{proof}[Proof of Proposition \ref{prop:unbiased estimators}]
Conditionally on $Z$, the index $J$ is uniform on $[N]\setminus Z$ and independent of $X$,
\begin{equation}\label{eq:bias1}
f(i)
=\mathbb E\big[\log \pi(X^J\mid X^Z)\big]
=\mathbb E\Big[\tfrac{1}{N-i}\textstyle\sum_{j\notin Z}\log\pi_j(X^j\mid X^Z)\Big]
=\mathbb E\big[\tilde f_i\big].    
\end{equation}
Conditionally on $(Z,X^Z)$, each unrevealed coordinate
$X^j$ has distribution $\pi_j(\cdot\mid X^Z)$. 
Therefore
\begin{equation}\label{eq:blackwell}
\begin{aligned}
\E[\tilde f_i\mid Z,X^Z]
&=
\frac{1}{N-i}\sum_{j\notin Z}
\E\!\left[
\log\pi_j(X^j\mid X^Z)
\,\middle|\, Z,X^Z
\right]\\
&=
-\frac{1}{N-i}\sum_{j\notin Z}
H\bigl(\pi_j(\cdot\mid X^Z)\bigr)
=
\hat f_i.
\end{aligned}
\end{equation}
Taking expectation with respect to $(Z,X^Z)$ and using the tower property together with \eqref{eq:bias1} yields
\[
\mathbb E\big[\hat f_i\big]
=\mathbb E \left[ \,\mathbb E\left[\tilde f_i\mid Z,X^Z\right]\right]
=\mathbb E\big[\tilde f_i\big]
=f(i). 
\]
Finally, the law of total variance applied to $\tilde f_i$ and
\eqref{eq:blackwell} give
\begin{align*}
\var(\tilde f_i)&=\E[\var(\tilde f_i\mid Z,X^Z)]+\var(\E[\tilde f_i\mid Z,X^Z]) \ge \var(\hat f_i).
\end{align*}
Thus, $\hat f_i$ is the Rao-Blackwellization of
$\tilde f_i$ with respect to $(Z,X^Z)$.
\end{proof}

\begin{proof}[Proof of Proposition \ref{prop:bias}]
Conditionally on $(Z,X^Z)$, write
\[
p_j:=\pi_j(\cdot\mid X^Z),
\qquad
q_j:=p_\theta^j(\cdot\mid X^Z),
\qquad j\notin Z,
\] 
so that
\[
\kl{p_j}{q_j}=-\E\bigl[\log q_j(X^j)\bigr]+\E\bigl[\log p_j(X^j)\bigr].
\]
Averaging over $j\notin Z$ and taking expectation over $(Z,X^Z)$, the second term is $f(i)$ by \eqref{eq:bias1} and the first is $-\E[\tilde f^\theta_i]$, giving $\E[\tilde f_i^\theta]-f(i)=-\mathcal L_i$. The sign comes from  the non-negativity of KL divergence.

For the entropy estimator, let $T_j:=\tfrac12\|p_j-q_j\|_1$. The Fannes-Audenaert inequality \citep{fannes1973continuity,audenaert2006sharp} gives
\begin{equation}\label{eq:aud}
|H(p_j)-H(q_j)|\le\phi(T_j).
\end{equation}
Since
\[
\E[\hat f_i^\theta]-f(i)
=
\E[H(p_J)-H(q_J)],
\]
and $\phi$ is concave on $[0,1]$, being the sum of a linear term and the concave binary entropy, Jensen's inequality yields
\[
\begin{aligned}
\bigl|\E[\hat f_i^\theta]-f(i)\bigr|
\le
\E\bigl[|H(p_J)-H(q_J)|\bigr]
\le
\E[\phi(T_J)]
\le
\phi(\E[T_J]).
\end{aligned}
\]
By Pinsker's inequality, the concavity of the square root and the assumption on $\mathcal{L}_i$,
\[
0\le\E[T_J]
\le
\E\!\left[\sqrt{\kl{p_J}{q_J}/2}\right]
\le
\sqrt{\E[\kl{p_J}{q_J}]/2}
=
\sqrt{\mathcal L_i/2}\le 1-1/|\x|
\]
As $\phi'(T)=\log\!\big(\tfrac{(|\x|-1)(1-T)}{T}\big)>0$ for $T<1-1/|\x|$, $\phi$ is nondecreasing on $(0,1-1/|\x|)$, so we conclude that
\[
\bigl|\E[\hat f_i^\theta]-f(i)\bigr|
\le
\phi(\E[T_J])
\le
\phi\!\left(\sqrt{\mathcal L_i/2}\right).
\]

For the rate, $H_2(T)=T\log(1/T)+O(T)$ as $T\to0$, so $\phi(T)=O(T\log(1/T))$. Substituting $T=\sqrt{\mathcal L_i/2}$ gives the stated rate.
\end{proof}

\begin{proof}[Proof of Proposition \ref{prop:estimator variance}]
Since $0\leq H(p)\leq \log |\x|$ for every $p\in\mathcal{P}(\x)$,  each $\hat f_{i,m}^{\theta}$ takes values in $[-\log |\x|,0]$. Popoviciu's inequality therefore gives
$\var \hat f_{i,m}^{\theta}\leq\frac{(\log|\x|)^2}{4}$   
and, by independence and identical distribution, 
\[
\var(\bar f_i)
=
\frac1n\var(\hat f_{i,1}^{\theta})
\le
\frac{(\log|\x|)^2}{4n}.
\]
Hoeffding's inequality for bounded independent summands gives, for every $t>0$,
\[
\PP\left(
\bigl|\bar f_i-\E[\hat f_i^\theta]\bigr|\ge t
\right)\le2\exp\Bigl(-\frac{2nt^2}{(\log|\x|)^2}\Bigr),
\]
and taking $t=\log|\x|\sqrt{\frac{\log(2/\delta)}{2n}}$ proves the claimed probability bound.
\end{proof}

\begin{proof}[Proof of Proposition \ref{prop:stability}]
For $i=0,\ldots,N-2$, set
\[
w_i(\beta):=N\sum_{k=1}^K\int_{\beta_{k-1}}^{\beta_k}
(\beta_k-u)b_i(u)\,du,
\]
 $b_i(u)=(N-1)B_i^{N-2}(u)$, as in \eqref{eq:rho-beta}.
The definition of $\rho_a$ gives \[
\efact(a,\beta)=\sum_{i=0}^{N-2}a(i)w_i(\beta),
\]
proving linearity. 
Since $b_i$ is nonnegative and integrates to one,
$
0\le w_i(\beta)\le N\Delta\beta_{\max}$, consequently,
\[
\left|\efact(\iota,\beta)-\efact(\bar\iota,\beta)\right|
\le\|\iota-\bar\iota\|_1\|w(\beta)\|_\infty
\le N\Delta\beta_{\max}\|\iota-\bar\iota\|_1,
\]
which proves \eqref{eq:first-ineq-efact}. 

For the optimization bound, applying \eqref{eq:first-ineq-efact} at $\hat\beta$ gives
\[
\efact(\iota,\hat\beta)
\le\efact(\bar\iota,\hat\beta)
+N\Delta\hat\beta_{\max}\|\iota-\bar\iota\|_1,
\]
while optimality of $\hat\beta$ and another application of \eqref{eq:first-ineq-efact},
\[
\efact(\bar\iota,\hat\beta)
\le\efact(\bar\iota,\beta^*)
\le\efact(\iota,\beta^*)
+N\Delta\beta^*_{\max}\|\iota-\bar\iota\|_1.
\]
Combining these inequalities yields
\[
\begin{aligned}
\efact(\iota,\hat\beta)-\efact(\iota,\beta^*)
\le
N\bigl(
\Delta_{\max}(\hat\beta)+\Delta_{\max}(\beta^*)
\bigr)
\|\iota-\bar\iota\|_1,
\end{aligned}
\]
which implies the stated upper bound.
The lower bound follows from the optimality of $\beta^*$
for the true profile.
\end{proof}

\subsection{Implementation details for the mixture target}\label{app:conditionals}

Let $U\subseteq[N]$ be the revealed set of coordinates with values $X^U$. 
By Bayes' rule, the posterior distribution of the latent variable is
\begin{equation}\label{eq:latent-posterior}
P(Z=z\mid x^U)=\frac{w_z\prod_{j\in U}\mu_{z,j}(x^j)}{\sum_{z'=1}^r w_{z'}\prod_{j\in U}\mu_{z',j}(x^j)}.
\end{equation}
Conditional independence given $Z$ then yields, for every $B\subseteq[N]\setminus U$,
\begin{align}
\pi(x^B\mid x^U)
&=\sum_{z=1}^r \PP(Z=z\mid X^U=x^U)\prod_{i\in B}\mu_{z,i}(x^i),\label{eq:block-conditional}
\\
\pi_i(x^i\mid x^U)
&=\sum_{z=1}^r \PP(Z=z\mid X^U=x^U)\,\mu_{z,i}(x^i),\qquad i\notin U.\label{eq:coord-conditional}
\end{align}

\paragraph{Estimating $\iota$ and $\rho$ through the auxiliary profile.}
We estimate the information profile $f$ using the entropy estimator in \eqref{eq:entropy-estimator}. 
For each Monte Carlo draw, we sample $X\sim\pi$ and, independently, a uniformly random permutation $\sigma$ of $[N]$. 
For $i=0,\ldots,N-1$, let 
$U_i=\{\sigma_1,\dots,\sigma_i\}$ 
be the set of coordinates revealed after $i$ steps. Conditional on $X^{U_i}$, the contribution to the estimate of $f(i)$ is 
\[
\hat f_i
=
-\frac{1}{N-i}
\sum_{j\notin U_i}
H\!\left(
\pi_j(\,\cdot\mid X^{U_i})
\right).
\]

Starting from the prior weights $\PP(Z=z\mid X^{U_0})=w_z$, we update the latent log-posterior after revealing coordinate $\sigma_{i+1}$ as
\[
\log \PP\left(
Z=z\mid X^{U_{i+1}}
\right)
=
\log \PP\left(
Z=z\mid X^{U_i}
\right)
+
\log\mu_{z,\sigma_{i+1}}
\left(X^{\sigma_{i+1}}\right)
+
\mathrm{const}.
\]

Thus, a single pass through the permutation produces the entire vector
$(\hat f_0,\ldots,\widehat f_{N-1})$.
Direct evaluation costs $O(N^2rL)$ per draw:
for each of the $N$ revealed-set sizes, we compute the conditional distributions and entropies of the remaining coordinates, at a cost of $O(rL)$ per coordinate.
We average the profiles over $40\,000$ independent draws.

Let $f^{\mathrm{MC}}$ denote the resulting raw average.
To enforce the monotonicity of the true profile, we compute its isotonic regression, that is its least-squares projection onto the set of non-decreasing sequences
\[
\bar f
:=
\arg\min_{v_0\le\cdots\le v_{N-1}}
\sum_{i=0}^{N-1}
\bigl(v_i-f_i^{\mathrm{MC}}\bigr)^2.
\]
We then set
\[
\bar\iota(i):=\bar f_{i+1}-\bar f_i\ge0,
\qquad i=0,\ldots,N-2,
\]
and construct $\bar\rho$ using \eqref{eq:bernstein_rep}.

\paragraph{Monte Carlo estimate of $\varepsilon_{\mathrm{fact}}$.}
Independently of the information-profile calculation, for each target and schedule $\beta$, we estimate $\varepsilon_{\mathrm{fact}}(\beta)$ directly from \eqref{eq:factorization-error} using $4000$ draws. 

For each draw, we sample $X\sim\pi$ and, following the multinomial coupling used in the proof of Theorem~\ref{thm:main}, independently assign each coordinate $i$ to a reveal step distributed as $\mathrm{Cat}(\Delta\beta_1,\dots,\Delta\beta_K)$, realizing  the distribution $\nu^\beta$. 
We then proceed block by block and accumulate the difference between the log block conditional in \eqref{eq:block-conditional} and the sum of the corresponding log one-coordinate conditionals in \eqref{eq:coord-conditional}.

\section{Code}    
The code for generating the simulated data and reproducing the numerical results is available at 
\url{https://github.com/ceciliasecchi/masked-tau-leaping}.

\end{document}